\documentclass[10pt]{article}
\usepackage{amssymb,amsmath,amsthm}
\usepackage{amsfonts}
\usepackage{graphicx,subfigure,epstopdf}
\usepackage[usenames]{color}
\usepackage{url}
\usepackage{algorithm,algorithmic}
\usepackage{dsfont,verbatim}
\usepackage[margin=1.2in]{geometry}
\allowdisplaybreaks

\newtheorem{theorem}{Theorem}
\newtheorem{remark}{Remark}
\newtheorem{definition}{Definition}
\newtheorem{lemma}{Lemma}

\newtheorem{example}{Example}[section]
\numberwithin{equation}{section}

\title{Discrete Maximal Regularity of Time-Stepping Schemes for Fractional Evolution Equations\thanks{The research of B. Jin
has been partly supported by UK EPSRC EP/M025160/1, and that of B. Li
was partially carried during a research stay at University of T\"{u}bingen, funded by the
Alexandre von Humboldt foundation. The research of Z. Zhou is supported in part by the AFOSR MURI Center for Material Failure Prediction
through Peridynamics and the ARO MURI Grant W911NF-15-1-0562.}
}

\author{Bangti Jin\thanks{Department of Computer Science, University College London, Gower Street, London WC1E 6BT, UK
(\texttt{b.jin@ucl.ac.uk})}
\and Buyang Li\thanks{Department of Applied Mathematics, The Hong Kong Polytechnic University, Hung Hom, Hong Kong.
(\texttt{buyang.li@polyu.edu.hk})}
\and Zhi Zhou\thanks{Department of Applied Physics and Applied Mathematics,
Columbia University, 500 W. 120th Street, New York, NY 10027, USA (\texttt{zz2393@columbia.edu})}
}
\date{\today}

\begin{document}

\maketitle

\begin{abstract}
In this work, we establish the maximal $\ell^p$-regularity for several time
stepping schemes for a fractional evolution model, which involves a
fractional derivative of order $\alpha\in(0,2)$, $\alpha\neq 1$, in time. These schemes include convolution
quadratures generated by backward Euler method and second-order backward difference formula,
the L1 scheme, explicit Euler method and a fractional variant of the Crank--Nicolson method. The main tools for the analysis include
operator-valued Fourier multiplier theorem due to Weis \cite{Weis:2001} and
its discrete analogue due to Blunck \cite{Blunck:2001}. These results generalize the corresponding results for parabolic problems.
\end{abstract}

\section{Introduction}
Maximal $L^p$-regularity is an important mathematical tool in studying the existence, uniqueness and
regularity of solutions of nonlinear partial differential equations of parabolic type. A generator $A$
of an analytic semigroup on a Banach space $X$ is said to have maximal $L^p$-regularity, if the
solution $u$ of the following parabolic differential equation
\begin{equation}\label{eqn:parabolic}
  \begin{aligned}
     u^\prime(t) & = A u + f\quad \forall t>0,\\
    u(0) & =0,
  \end{aligned}
\end{equation}
satisfies the following estimate
\begin{equation}\label{eqn:parabolic-maximal-reg}
  \| u^\prime \|_{L^p(\mathbb{R}^+;X)} + \|Au\|_{L^p(\mathbb{R}^+;X)}\leq c_{p,X}\|f\|_{L^p(\mathbb{R}^+;X)}\quad \forall f\in L^p(\mathbb{R}^+;X),
\end{equation}
with $1<p<\infty$. On a Hilbert space $X$, every generator of a bounded analytic semigroup has maximal
$L^p$-regularity \cite{deSimon:1964}, and Hilbert spaces are only spaces for which
this holds true \cite{KaltonLancien:2000}. Beyond Hilbert spaces, an important
and very useful characterization of the maximal $L^p$-regularity was given by Weis \cite{Weis:2001} on
UMD spaces in terms of the $R$-boundedness of a family of operators using the resolvent $R(z;A):=(z-A)^{-1}$;
see Theorem \ref{thm:Weis} in Section \ref{sec:prelim} for details.

An important question from the perspective of numerical analysis is whether such maximal regularity estimates
carry over to time-stepping schemes for discretizing the parabolic problem \eqref{eqn:parabolic},
which have important applications in numerical analysis of nonlinear parabolic problems
\cite{AkrivisLi2017,AkrivisLiLubich2016,Geissert:2007,KLC2016,LiSun2015-Regularity}.
This question has been studied in a number of works from different aspects
\cite{AshyralyevSobolevskii:1994,AshyralyevPiskarevWeis:2002,Geissert:2006,Geissert:2006b,LeykekhmanVexler:2017,Li2015,LiSun2017-MCOM}.
Ashyralyev, Piskarev and Weis
\cite{AshyralyevPiskarevWeis:2002} showed the following discrete maximal regularity: for all $f^n\in X,\,\,
  n=1,2,\dots$,
\begin{equation*}
  \tau^{-1}\| (u^n-u^{n-1})_{n=1}^N\|_{\ell^p(X)} + \|(Au^n)_{n=1}^N\|_{\ell^p(X)}\leq c_{p,X}\|(f^n)_{n=1}^N\|_{\ell^p(X)}
\end{equation*}
for the time-discrete solutions $u^n$, $n=1,2,\dots,$ given by the implicit Euler method, where $\tau$ is
the time step size and the constant $c_{p,X}$ is independent of $\tau$.
A variant of the maximal $\ell^p$-regularity for the Crank-Nicolson method was also shown in
\cite{AshyralyevPiskarevWeis:2002}. Recently, Kov\'{a}cs, Li and Lubich \cite{KovacsLiLubich:2015}
proved the discrete maximal regularity for the Crank-Nicolson, BDF and $A$-stable Runge--Kutta methods. Kemmochi and Saito \cite{Kemmochi:2015,KS2016} proved the maximal $\ell^p$-regularity for the $\theta$-method. In these works,
the main tools are the maximal $L^p$-regularity characterization due to Weis \cite{Weis:2001} and its discrete analogue due to Blunck \cite{Blunck:2001}.
Independently, Leykekhman and Vexler \cite{LeykekhmanVexler:2017}
proved the maximal $L^p$-regularity of discontinuous Galerkin methods without using Blunck's multiplier technique. The maximal $\ell^p$-regularity of fully discrete numerical solutions have been investigated in \cite{Kemmochi:2015,KS2016,LeykekhmanVexler:2017} and \cite{LiSun2017-SINUM} for parabolic equations with time-independent and time-dependent coefficients, respectively; also see \cite[section 6]{KovacsLiLubich:2015}.

The maximal $L^p$-regularity has also been studied for the following fractional evolution equation
\begin{equation}\label{eqn:fde}
    \partial_t^\alpha u(t) = A u(t) + f\quad\forall t>0,\\
\end{equation}
together with the following initial condition(s)
\begin{equation*}
  \begin{aligned}
    u(0) & = 0,  && \mbox{if} \quad 0<\alpha<1,\\
    u(0)  &= 0, \quad \partial_t u (0) = 0,
    && \mbox{if}\quad  1<\alpha<2.
  \end{aligned}
\end{equation*}
In the model \eqref{eqn:fde},
the notation $\partial_t^\alpha u$ denotes the Caputo fractional derivative of order $\alpha$
of $u$ with respect to time $t$, defined by  \cite[pp. 91]{KilbasSrivastavaTrujillo:2006}
\begin{equation*}
\partial_t^\alpha u(t)=
\frac{1}{\Gamma(n-\alpha)}  \int_0^t(t-s)^{n -\alpha-1}\frac{d^n}{ds^n}u(s)ds,
\quad n-1<\alpha<n, \, n\in\mathbb{N},
\end{equation*}
where the Gamma function $\Gamma(\cdot)$ is defined by $\Gamma(z)=\int_0^\infty s^{z-1}e^{-s}ds$,
$\Re z>0$. With zero initial condition(s), it is identical with the Riemann-Liouville one
\cite[pp. 70]{KilbasSrivastavaTrujillo:2006}
$$
^R\kern -.5mm\partial_t^{\alpha} u(t)=\frac{1}{\Gamma(n-\alpha)}
\frac{d^n}{dt^n}\int_0^t(t-s)^{n-\alpha-1}u(s)ds,
\quad n-1<\alpha<n, \, n\in\mathbb{N}.
$$
Throughout, we use only the notation $\partial_t^\alpha u$ to denote either derivative.
When $\alpha=1$, the fractional derivative $\partial_t^\alpha u(t)$ coincides with the usual first-order derivative
$u^\prime(t)$, and accordingly, the fractional model \eqref{eqn:fde} recovers the standard parabolic equation \eqref{eqn:parabolic}.
In this paper we focus on the fractional cases $0 <
\alpha < 1$ and $1 < \alpha < 2$, which are known as the subdiffusion
and diffusion-wave equation, respectively. In analogy with Brownian
motion for normal diffusion \eqref{eqn:parabolic}, the model \eqref{eqn:fde} with $0 <\alpha < 1$ is the macroscopic
counterpart of continuous time random walk.

The fractional model \eqref{eqn:fde} has received much attention in recent years, since it can adequately capture the
dynamics of anomalous diffusion processes. For example, the subdiffusion equation, i.e., $\alpha\in(0,1)$, has been
employed to describe transport in column experiments, thermal diffusion in media with fractal geometry, and
flow in highly heterogeneous aquifers. See \cite{MetzlerKlafter:2000} for an extensive list of applications.
 The diffusion-wave equation, i.e., $\alpha\in(1,2)$, can be used to model mechanical wave propagation in viscoelastic media.

In a series of interesting works \cite{Bajlekov:2001,Bazhlekova:2002,BazhlekovaClement:2003}, Bazhalekov and collaborators have established the following maximal $L^p$-regularity
for the fractional model \eqref{eqn:fde}: for any $1<p<\infty$, $u\in L^p(\mathbb{R}^+;D(A))$ and
\begin{equation}\label{eqn:maximal-fde}
  \|\partial_t^\alpha u \|_{L^p(\mathbb{R}^+;X)} + \|Au\|_{L^p(\mathbb{R}^+;X)}\leq c_{p,X}\|f\|_{L^p(\mathbb{R}^+;X)}\quad \forall f\in L^p(\mathbb{R}^+;X),
\end{equation}
 under suitable conditions on the operator $A$ (see Theorem \ref{thm:maximal-fde} in Section
\ref{sec:prelim} for details). Further, they applied the theory to analyze nonautonomous and semilinear problems
\cite{Bajlekov:2001,BazhlekovaClement:2003}. See also \cite{Pruss:1993} for
closely related maximal regularity results for Volterra integro-differential equations.

The discrete analogue of \eqref{eqn:maximal-fde} is important for the numerical analysis of
nonautonomous and nonlinear fractional evolution problems.  The only existing result we are
aware of is the very recent work of Lizama \cite{Lizama:2015}.
Specifically, Lizama studied the following fractional difference equation with $0<\alpha<1$:
$$
\Delta^\alpha u^n=Tu^n+f^n ,
$$
where $u^0=0$ and $\Delta^\alpha$ is a certain fractional difference operator. The author established
the maximal $\ell^p$-regularity for the problem, under the condition that
the set $\{\delta(z)(\delta(z)-T)^{-1}:|z|=1,z\neq1\}$ is $R$-bounded, with $\delta(z)=z^{1-\alpha}(1-z)^\alpha$,
following the work of Blunck \cite{Blunck:2001}.
It can be interpreted as a time-stepping
scheme: upon letting $T=\tau^\alpha A$ and $f^n=\tau^\alpha g^n$, we get $\tau ^{-\alpha} \Delta^\alpha u^n= Au^n+g^n$. Hence, it
amounts to a convolution quadrature generated by the kernel  $z^{1-\alpha}(1-z)^\alpha$.
However, this scheme lacks the maximal $\ell^p$-regularity if $A=\Delta$, the Dirichlet Laplacian operator in bounded domains.

In this work, we address the following question: \textit{Under which conditions do the time discretizations
of \eqref{eqn:fde} preserve the maximal $\ell^p$-regularity, uniformly in the step size $\tau$?} We provide an
analysis for several time-stepping schemes, including the convolution quadratures generated by the implicit
Euler method and second-order backward difference formula \cite{CuestaLubichPalencia:2006,JinLazarovZhou:2016sisc},
the L1 scheme \cite{LinXu:2007,SunWu:2006},
the explicit Euler method \cite{YusteAcedo:2005} and a fractional variant of the Crank--Nicolson method.
Amongst them, the convolution quadrature is relatively easy to analyze.
In contrast, the L1 scheme and explicit Euler method are easy to implement, but challenging to analyze. The explicit Euler method requires a bounded numerical range of the operator $A$ and the step size $\tau$ to be small enough. The maximal $\ell^p$-regularity of the Crank--Nicolson method behaves like the implicit Euler scheme when $0<\alpha<1$ and like the explicit Euler scheme when $1<\alpha<2$.
Our proof strategy follows closely the recent
works \cite{KovacsLiLubich:2015,Kemmochi:2015} and employs the (discrete) Fourier multiplier technique of Blunck \cite{Blunck:2001}.

The rest of the paper is organized as follows. In Section \ref{sec:prelim} we recall basic tools
for showing maximal $\ell^p$-regularity, including $R$-boundedness, UMD spaces, and Fourier multiplier
theorems. Then four classes of time-stepping schemes, i.e., convolution quadrature, L1 scheme, explicit
Euler method and a variant of the fractional Crank-Nicolson method, are discussed in Sections
\ref{sec:CQ}-\ref{sec:CN}, respectively. In Section \ref{sec:inhomo}, we discuss the extension
to nonzero initial data. Last, in Section \ref{sec:example} we illustrate the results
with several concrete examples.

We conclude the introduction with some notation. For a Banach space $X$, we
denote by $\mathcal{B}(X)$ the set of all bounded linear operators from $X$ into itself. For a linear
operator $A$ on $X$, we denote by $\sigma(A)$ and $\rho(A)$ its spectrum and resolvent set, respectively. We denote
the unit circle in the complex plane $\mathbb{C}$ by $\mathbb{D}=\{z:|z|=1\}$, and $\mathbb{D}^\prime=\{z:|z|=1,z\neq\pm1\}$.
Given any $\theta\in(0,\pi),$ the notation $\Sigma_\theta$ denotes the open  sector
$
  \Sigma_\theta:=\left\{z\in\mathbb{C}: |\arg(z)|< \theta, z\neq0\right\} ,
$
where ${\rm arg}(z)$ denotes the argument of $z\in{\mathbb C}\backslash\{0\}$ in the range $(-\pi,\pi]$.
Throughout, the notation $c$ and $C$, with or without a subscript/superscript, denote a generic constant, which may
differ at different occurrences, but it is always independent of the time step size $\tau$ and the number $N$ of time steps.

\section{Preliminaries}\label{sec:prelim}
In this section we collect basic results on the maximal $L^p$-regularity and related concepts, especially
$R$-boundedness, UMD spaces, and Fourier multiplier theorems, used in the fundamental
work of Weis \cite{Weis:2001}, where he characterized the maximal $L^p$-regularity of an operator $A$ in
terms of its resolvent operator $R(z;A):=(z-A)^{-1}$. We refer readers to the review
\cite{KunstmannWeis:2004} for details.

\subsection{$R$-boundedness}

The concept of $R$-boundedness plays a crucial role in Weis' operator-valued Fourier multiplier theorem and
its discrete analogue. A collection of operators $\mathcal{M}=\{M(\lambda): \lambda\in \Lambda\}\subset \mathcal{B}(X)$ is said to be
$R$-bounded if there is a constant $c>0$ such that any finite subcollection of operators
$\{M(\lambda_j)\}_{j=1}^l$ satisfies
\begin{equation}\label{eqn:R-bound}
  \int_0^1\bigg\|\sum_{j=1}^lr_j(s)M(\lambda_j)v_j\bigg\|_Xds\leq c\int_0^1\bigg\|\sum_{j=1}^lr_j(s)v_j\bigg\|_Xds\quad \forall v_1,v_2,\ldots,v_l\in X,
\end{equation}
where $r_j(s)=\mathrm{sign}\sin(2j\pi s)$, $j=1,2,\ldots,$ are the Rademacher functions defined
on the interval $[0,1]$. The infimum of the constant $c$ satisfying \eqref{eqn:R-bound}, denoted
by $R(\mathcal{M})$ below, is called the $R$-bound of the set $\mathcal{M}$.
In particular, if $\Lambda\subset\{z\in\mathbb{C}: |z|\leq c_0\}$
for some $c_0>0$, then the set $\{\lambda I: \lambda\in \Lambda\}$ is $R$-bounded with an $R$-bound $2c_0$.
This fact will be used extensively below.

There are a number of basic properties of $R$-bounded sets, summarized below. They follow
from definition and the proofs can be found in \cite{KunstmannWeis:2004}.
\begin{lemma}\label{lem:RB-calculus}
Let $\mathcal{T}\subset \mathcal{B}(X)$ be an $R$-bounded set. Then the following statements hold.
\begin{itemize}
  \item[$(\mathrm{i})$] If $\mathcal{S}\subset \mathcal{T}$, then $\mathcal{S}$ is $R$-bounded with $R(\mathcal{S})\leq R(\mathcal{T})$.
  \item[$(\mathrm{ii})$] The closure $\overline{\mathcal{T}}$ in $\mathcal{B}(X)$ is also $R$-bounded, and $R(\overline{\mathcal{T}})=R({\cal T})$.
  \item[$(\mathrm{iii})$] If $\mathcal{S}\subset \mathcal{B}(X)$ is $R$-bounded, then the union ${\cal S}\cup {\cal T}$ and sum ${\cal S+T}$ are
   $R$-bounded, with bounds $R({\cal S\cup T})\leq R({\cal S})+R({\cal T})$ and $R({\cal S+T})\leq R({\cal S})+R({\cal T})$.
  \item[$(\mathrm{iv})$] If ${\cal S}\subset\mathcal{B}(X)$ is $R$-bounded, then ${\cal ST}$ is $R$-bounded with $R({\cal ST})=R({\cal S})R({\cal T})$.
  \item[$(\mathrm{v})$] The convex hull $\mathrm{CH}({\cal T})$ is $R$-bounded with $R(\mathrm{CH}({\cal T}))\leq R({\cal T})$.
  \item[$(\mathrm{vi})$] The absolutely convex hull of ${\cal T}$, denoted by $\mathrm{ACH}_\mathbb{C}({\cal T})$, is $R$-bounded, with $R(\mathrm{ACH}_\mathbb{C}({\cal T}))\leq 2R({\cal T})$.
\end{itemize}
\end{lemma}

The following useful result is a slight extension of \cite[Corollary 3.5]{Blunck:2001}.
\begin{lemma}\label{lem:Blunk}
Let $A$ be a closed and densely defined operator in $X$, and $\delta\in(0,\pi)$. If $\{zR(z;A):z\in\Sigma_{\delta}\}$ is $R$-bounded, then there exists $\delta^\prime\in(\delta,\pi)$
such that $\{zR(z;A): z\in \Sigma_{\delta^\prime}\}$ is $R$-bounded.
\end{lemma}
\begin{proof}$\,\,$
In fact, the $R$-boundedness of $\{zR(z;A):z\in \Sigma_{\delta}\}$ implies the $R$-boundedness
of $\{\rho e^{\mathrm{i}\delta} R(\rho e^{\mathrm{i}\delta};A):\rho>0\}$. Via a rotation in
the complex plane $\mathbb{C}$, we see that $\{\rho\mathrm{ i} R(\rho \mathrm{i} ;e^{{\rm i}(\pi/2-\delta)}A):
\rho>0\}$ is $R$-bounded. Then the proof of \cite[Corollary 3.5]{Blunck:2001}
implies the $R$-boundedness of $\{w R(w;e^{{\rm i}(\pi/2-\delta)}A): \pi/2<\arg(w)<\pi/2+\vartheta\}$ for some small $\vartheta>0$.
By rotating back in the complex plane $\mathbb{C}$, we have the $R$-boundedness of $\{z R(z;A): \delta<\arg(z)<\delta+\vartheta\}$.
The $R$-boundedness of $\{z R(z;A): -\delta-\vartheta<\arg(z)<-\delta\}$ follows similarly. Overall, the set
 $\{z R(z;A): z\in\Sigma_{\delta+\vartheta}\}$ is $R$-bounded.
\qed\end{proof}

\subsection{Operator-valued multiplier theorems on UMD spaces}
Now we recall the concept of UMD spaces, which is essential for multiplier theorems.
Let $S(\mathbb{R};X)$ denote the space of rapidly decreasing $X$-valued functions. A Banach space
$X$ is said to be a UMD space if the Hilbert transform
\begin{equation*}
  Hf(t)=\mathrm{P.V.}\int_\mathbb{R}\frac{1}{t-s}f(s)ds,
\end{equation*}
defined on the space $S(\mathbb{R};X)$, can be extended to a bounded operator on $L^p(\mathbb{R};X)$
for all $1<p<\infty$. Equivalently, this can be characterized
by unconditional martingale differences, hence the abbreviation UMD. Examples of UMD spaces
include Hilbert spaces, finite-dimensional Banach spaces, and $L^q(\Omega,d\mu)$ ($(\Omega,\mu)$
is a $\sigma$-finite measure space, $1<q<\infty$), and its closed subspaces (e.g., Sobolev spaces
$W^{m,p}(\Omega)$, $1<p<\infty$), and the product space of UMD spaces. Throughout, $X$ always
denotes a UMD space. Next we recall the concept of $R$-sectoriality operator.
The definition below is equivalent to \cite[Section 1.11]{KunstmannWeis:2004}
by changing $A$ to $-A$ and changing $\theta$ to $\pi-\theta$.

\begin{definition}\label{Def-R-sectorial}
An operator $A:D(A)\rightarrow X$ is said to be sectorial of angle $\theta$ if the
following three conditions are satisfied:
\begin{itemize}
\item[$\mathrm{(i)}$] $A:D(A)\rightarrow X$ is a closed operator and its domain $D(A)$ is dense in $X$;
\item[$\mathrm{(ii)}$] The spectrum of $A$ is contained in
the sector ${\mathbb C}\backslash \Sigma_{\theta}$;
\item[$\mathrm{(iii)}$] The set of operators $\{zR(z;A) :z\in \Sigma_{\theta}\}$ is bounded in $\mathcal{B}(X)$.
\end{itemize}
Similarly, $A$ is said to be $R$-sectorial of angle
$\theta$ if (i), (ii) and the following condition hold:
\begin{itemize}
\item[$\mathrm{(iii}')$] The set of operators $\{zR(z;A) :z\in \Sigma_{\theta}\}$ is $R$-bounded in $\mathcal{B}(X)$.
\end{itemize}
\end{definition}

The following theorem is a simple consequence of
Dore \cite[Theorem 2.1]{Dore:1991} and
Weis \cite[Theorem 4.2]{Weis:2001}.

\begin{theorem}\label{thm:Weis}
A densely defined closed operator $A$ on a UMD space $X$ has maximal parabolic $L^p$-regularity \eqref{eqn:parabolic-maximal-reg} if and only if $A$ is $R$-sectorial of angle $\pi/2$.
\end{theorem}

The ``if'' direction in Theorem \ref{thm:Weis} is a consequence of
the following operator-valued Fourier multiplier theorem \cite[Theorem 3.4]{Weis:2001},
where $\mathcal{F}$ denotes the Fourier transform on $\mathbb{R}$, i.e.,
\begin{equation*}
  \mathcal{F}f(\xi) = \int_\mathbb{R}e^{-\mathrm{i}\xi t}f(t)dt\quad \xi\in\mathbb{R}.
\end{equation*}
\begin{theorem}\label{thm:Fourier-multiplier}
Let $X$ be a UMD space. Let $M:\mathbb{R}\setminus\{0\}\to \mathcal{B}(X)$ be differentiable such that the set
\begin{equation*}
  \{M(\xi):\xi\in\mathbb{R}\setminus\{0\}\}\cup\{\xi M^\prime(\xi):\xi\in\mathbb{R}\setminus\{0\}\}\mbox{ is $R$-bounded},
\end{equation*}
with an $R$-bound $c_R$. Then $\mathcal{M}f:=\mathcal{F}^{-1}(M(\cdot)(\mathcal{F}f)(\cdot))$ extends to a bounded operator
\begin{equation*}
  \mathcal{M}: L^p(\mathbb{R},X)\to L^p(\mathbb{R},X)\quad  \mbox{for } 1<p<\infty.
\end{equation*}
Further, there exists $c_{p,X}>0$ independent of $M$ such that the operator norm of $\mathcal{M}$ is bounded by $c_Rc_{p,X}$.
\end{theorem}

Using Theorem \ref{thm:Fourier-multiplier}, one can similarly show the following maximal
regularity result for the fractional model \eqref{eqn:fde} \cite{Bajlekov:2001,Bazhlekova:2002,BazhlekovaClement:2003},
which naturally extends the ``if'' part of Theorem \ref{thm:Weis} to the fractional case.

\begin{theorem}\label{thm:maximal-fde}
Let $A$ be an $R$-sectorial operator of angle $\alpha\pi/2$ on a UMD space $X$.
Then the solution of \eqref{eqn:fde} satisfies the maximal $L^p$-regularity
estimate \eqref{eqn:maximal-fde} for any $1<p<\infty$.
\end{theorem}

In this work, we discuss the discrete analogue of Theorem \ref{thm:maximal-fde} for
a number of time-stepping schemes for solving \eqref{eqn:fde}, under the same condition on the operator
$A$, using a discrete version of Theorem
\ref{thm:Fourier-multiplier} due to Blunck \cite{Blunck:2001}.
We slightly abuse $\mathcal{F}$ for the Fourier transform on $\mathbb{Z}_+:=\{n\in{\mathbb Z}: n\ge 0\}$, which
maps a sequence $(f^n)_{n=0}^\infty$ to its Fourier series on the interval $(0,2\pi)$, i.e.,
\begin{equation*}
  \mathcal{F}f(\theta ) = \sum_{n=0}^\infty e^{-\mathrm{i}n\theta }f^n \, ,\quad \forall\, \theta \in (0,2\pi),
\end{equation*}
and let ${\mathcal F}_\theta^{-1}$ denote the inverse Fourier transform
with respect to $\theta $, i.e.,
\begin{equation*}
{\mathcal F}_\theta^{-1}f(\theta) =
\Big(\frac{1}{2\pi}\int_0^{2\pi} f(\theta)e^{\mathrm{i} n\theta} d\theta \Big)_{n=0}^\infty \, .
\end{equation*}

The following result is an immediate consequence of \cite[Theorem 1.3]{Blunck:2001}, and will be used extensively;
see also \cite{Kemmochi:2015} for a proof with a more explicit constant. The statement
is equivalent to Blunck's original theorem via the transformation $\xi=e^{-{\rm i}\theta}$, but avoids introducing
a different notation $\widetilde M(\theta)$.
\begin{theorem}
\label{thm:blunk}
Let $X$ be a UMD space, and let $M:{\mathbb D}'\to \mathcal{B}(X)$ be
differentiable such that the set
\begin{equation}\label{eqn:set-Blunck}
  \left\{ M(\xi):\xi\in {\mathbb D}'\right\} \cup \left\{(1-\xi) (1+\xi) M^\prime(\xi): \xi\in {\mathbb D}'\right\}
\end{equation}
is $R$-bounded, with an $R$-bound $c_R$.
Then $\mathcal{M}f:=\mathcal{F}_\theta^{-1}( M(e^{-\mathrm{i} \theta})
(\mathcal{F}f)(\theta))$ extends to a bounded operator
\begin{equation*}
  \mathcal{M}:\ell^p(\mathbb{Z}_+,X)\to \ell^p(\mathbb{Z}_+,X)\quad \mbox{for } \,1<p<\infty.
\end{equation*}
Further, there exists a $c_{p,X}>0$ independent of $ M$ such that the operator norm of $\mathcal{M}$ is bounded by $c_Rc_{p,X}$.
\end{theorem}

To simplify the notations,
for a given sequence $\{M_n\}_{n=0}^\infty$ of operators on a UMD space $X$, we define the generating function
\begin{equation}
  M(\xi): = \sum_{n=0}^\infty M_n\xi^n\quad \forall \xi\in{\mathbb D}' .
\end{equation}
Likewise, the generating function $f(\xi)$ of a sequence $(f^n)_{n=0}^\infty$ is defined by
\begin{equation}\label{def-gener}
f(\xi):=\sum_{n=0}^\infty f^n\xi^n.
\end{equation}
The operator $\mathcal{M}$ is then given by
$(\mathcal{M}f)_n=\sum_{j=0}^nM_{n-j}f^j$, $n=0,1,\ldots$
The generating function satisfies the
convolution rule
\begin{align}\label{conv-rule}
(f\ast g)(\xi)=f(\xi)g(\xi) ,
\end{align}
where $
(f\ast g)_n:=\sum_{j=0}^n f_jg_{n-j},$ $n= 0 ,1,\dots $

\section{Convolution quadrature}\label{sec:CQ}

The convolution quadrature of Lubich (see the review \cite{Lubich:2004} and references therein) presents
one versatile framework for developing time-stepping schemes for the
model \eqref{eqn:fde}. One salient feature is that it inherits
excellent stability property (of that for ODEs). We shall consider convolution quadrature generated by backward Euler (BE)
 and second-order backward difference formula (BDF2),
whose error analysis has been carried out in \cite{CuestaLubichPalencia:2006,JinLazarovZhou:2016sisc}.

\subsection{BE scheme}
We first illustrate basic ideas to prove discrete maximal regularity on BE
scheme in time $t$, with a constant time step size $\tau>0$. The BE scheme for
\eqref{eqn:fde} is given by: given $u^0=0$, find $u^n\in X$
\begin{equation}\label{eqn:Euler}
  \bar\partial_\tau^\alpha u^n = A u^n + f^n,\quad n=1,2,\ldots
\end{equation}
where the BE approximation $\bar\partial_\tau^\alpha u^n$ to $\partial_t^\alpha u(t_n)$ is given by
\begin{equation}\label{eqn:BE}
  \bar\partial_\tau^\alpha u^n = \tau^{-\alpha}\sum_{j=0}^nb_{n-j}u^j\quad \mbox{with }
\sum_{j=0}^\infty b_j\xi^j=\delta(\xi)^\alpha:=(1-\xi)^\alpha,
\end{equation}
where $\delta(\xi)=1-\xi$ is the characteristic function of the BE method.

Now we can state the discrete maximal regularity of the BE scheme \eqref{eqn:Euler}.
\begin{theorem}\label{thm:Euler}
Let $X$ be a UMD space, $0<\alpha<1$ or $1<\alpha<2$, and let $A$ be an $R$-sectorial operator on $X$ of angle $\alpha\pi/2$.
Then the BE scheme
\eqref{eqn:Euler} has the following maximal $\ell^p$-regularity
\begin{equation*}
  \|(\bar\partial_\tau^\alpha u^n)_{n=1}^N\|_{\ell^p(X)}+ \|(Au^n)_{n=1}^N\|_{\ell^p(X)}\leq c_{p,X}c_R\|(f^n)_{n=1}^N\|_{\ell^p(X)},
\end{equation*}
where the constant $c_{p,X}$ is independent of $N$, $\tau$ and $A$,
and $c_R$ denotes the $R$-bound of the set of operators
$\{zR(z;A) :z\in \Sigma_{\alpha\pi/2}\}$.
\end{theorem}
\begin{proof}$\,\,$
By multiplying both sides of \eqref{eqn:Euler} by $\xi^n$ and summing over $n$, we have
\begin{equation*}
  \sum_{n=1}^\infty \xi^n \bar\partial_\tau^\alpha u^n - \sum_{n=1}^\infty Au^n\xi^n = \sum_{n=1}^\infty f^n\xi^n.
\end{equation*}
It suffices to compute the term $\sum_{n=1}^\infty \xi^n \bar\partial_\tau^\alpha u^n$. By noting
$u^0=0$, the definition of the BE approximation \eqref{eqn:BE} and discrete convolution rule
\eqref{conv-rule}, we deduce
\begin{equation*}
  \begin{aligned}
  \sum_{n=1}^\infty \xi^n \bar\partial_\tau^\alpha u^n & = \tau^{-\alpha}\sum_{n=0}^\infty \xi^n \sum_{j=0}^nb_{n-j}u^j = \tau^{-\alpha}\Big(\sum_{n=0}^\infty u^n\xi^n\Big)\Big(\sum_{n=0}^\infty b_n\xi^n\Big)\\
  &=\tau^{-\alpha}\delta(\xi)^\alpha u(\xi).
  \end{aligned}
\end{equation*}
Consequently, upon letting $f^0=0$, we arrive at
\begin{equation*}
  (\tau^{-\alpha}\delta_\tau(\xi)^\alpha -A) u(\xi) = f(\xi).
\end{equation*}
Since $\tau^{-1}\delta(\xi)\in\Sigma_{\pi/2}$ for $\xi\in{\mathbb D}'$, we have
$\tau^{-\alpha}\delta(\xi)^\alpha\in\Sigma_{\alpha\pi/2}$ for $\xi\in{\mathbb D}'$.
The $R$-sectoriality of angle $\alpha\pi/2$ of the operator $A$
ensures that the operator $\tau^{-\alpha} \delta(\xi)^\alpha-A$ is invertible.
Meanwhile, the generating function of the BE approximation $\bar\partial_\tau^\alpha u$ is given by
\begin{align*}
(\bar\partial_\tau^\alpha u)(\xi) &= \sum_{n=0}^\infty \xi^n\bar\partial_\tau^\alpha u^n =\tau^{-\alpha}\delta(\xi)^\alpha u(\xi)
= M(\xi)f(\xi).
\end{align*}
with $M(\xi) =\tau^{-\alpha}\delta(\xi)^\alpha(\tau^{-\alpha}\delta(\xi)^\alpha-A)^{-1}$.
Appealing to the $R$-sectoriality of $A$ again gives that $zR(z;A)$ is analytic and $R$-bounded
in the sector $\Sigma_{\alpha\pi/2}$, which imply that $M(\xi)$ is differentiable and $R$-bounded
for $\xi\in{\mathbb D}'$. Direct computation yields
\begin{equation*}
  (1-\xi)M^\prime(\xi) = - \alpha M(\xi) + \alpha M(\xi)^2,
\end{equation*}
which together with Lemma \ref{lem:RB-calculus}  (iii)-(iv) implies the $R$-boundedness of the set \eqref{eqn:set-Blunck}.
Then the desired result follows from Theorem \ref{thm:blunk}.
\qed\end{proof}

\begin{remark}
The BE scheme \eqref{eqn:BE} is identical with the Gr\"{u}nwald-Letnikov formula, a popular difference
analogue of the Riemann-Liouville fractional derivative $\partial_t^\alpha u$ \cite{Podlubny:1999}, which has been customarily
employed for discretizing \eqref{eqn:fde}.
\end{remark}

\subsection{Second-order BDF scheme}
Next we consider the convolution quadrature generated by the second-order backward difference formula (BDF2)
for discretizing the model \eqref{eqn:fde}:
\begin{equation}\label{eqn:SBD}
  \bar\partial_\tau^\alpha u^n = Au^n + f^n,\quad n\geq 2
\end{equation}
where the BDF2 approximation $\bar\partial_\tau^\alpha u^n$ to $\partial_t^\alpha u(t_n)$, $t_n=n\tau$, is given by
\begin{equation}\label{eqn:SBD_dis}
  \bar\partial_\tau^\alpha u^n = \tau^{-\alpha}\sum_{j=0}^n b_{n-j}u^{j}\quad
\mbox{with } \sum_{j=0}^\infty b_j\xi^j=\delta(\xi)^\alpha,
\end{equation}
with the characteristic function $\delta(\xi)$
\begin{equation}\label{eqn:delta_SBD}
  \delta(\xi) = \tfrac{3}{2}-2\xi+\tfrac{1}{2}\xi^2.
\end{equation}
We approximate the fractional derivative $\partial_t^\alpha u(t_n)$ by the BDF2 convolution quadrature \eqref{eqn:SBD_dis}, and
consider the scheme \eqref{eqn:SBD} with zero starting values $u^0=u^1=0$.
Note that for the BDF2 scheme (and other higher-order
linear multistep methods), the initial steps have to be corrected properly in order to achieve the desired
accuracy \cite{CuestaLubichPalencia:2006,JinLazarovZhou:2016sisc}.
The next result gives the discrete maximal regularity of the scheme \eqref{eqn:SBD}.
\begin{theorem}\label{thm:SBD}
Let $X$ be a UMD space, $0<\alpha<1$ or $1<\alpha<2$, and let $A$ be an $R$-sectorial operator on $X$ of angle $\alpha\pi/2$.
Then the BDF2  scheme \eqref{eqn:SBD}
with a step size $\tau$ satisfies the following discrete maximal regularity
\begin{equation*}
  \|(\bar\partial_\tau^\alpha u^n)_{n=2}^N\|_{\ell^p(X)} +\|(Au^n)_{n=2}^N\|_{\ell^p(X)}\leq c_{p,X}c_R\|(f_n)_{n=2}^N\|_{\ell^p(X)},
\end{equation*}
where the constant $c_{p,X}$ is independent of $N$, $\tau$ and $A$,
and $c_R$ denotes the $R$-bound of the set of operators
$\{zR(z;A) :z\in \Sigma_{\alpha\pi/2}\}$.
\end{theorem}
\begin{proof}$\,\,$
In a straightforward manner, upon letting $f^0=f^1=0$, we obtain
\begin{equation*}
  (\tau^{-\alpha}\delta(\xi)^\alpha-A)u(\xi) = f(\xi),
\end{equation*}
where $\delta(\xi)$ is defined in \eqref{eqn:delta_SBD}. Since BDF2 is $A$-stable (for ODEs), i.e.,
$\Re\delta(\xi)> 0$ for $\xi\in\mathbb{D}^\prime$, we have
$\tau^{-\alpha}\delta(\xi)^\alpha\subset \Sigma_{\alpha\pi/2}$. This
and the $R$-sectoriality (of angle $\alpha\pi/2$)
of the operator $A$ implies that $\tau^{-\alpha}\delta(\xi)^\alpha-A$ is invertible for $\xi\in \mathbb{D}'$.
Further, direct computation gives
\begin{equation*}
  (\bar\partial_\tau^\alpha u)(\xi)=M(\xi)f(\xi)\quad \mbox{with } M(\xi) = \tau^{-\alpha}\delta(\xi)^\alpha(\tau^{-\alpha}\delta(\xi)^\alpha-A)^{-1}.
\end{equation*}
The $R$-sectoriality of the operator $A$ implies the $R$-boundedness of the
set $\{M(\xi): \xi\in\mathbb{D}^\prime\}$. Meanwhile,
\begin{equation*}
  \begin{aligned}
    (1-\xi)M^\prime(\xi) 
     & = d(\xi) M(\xi) - d(\xi) M(\xi)^2,\quad \mbox{with $d(\xi)=\alpha \frac{2(\xi-2)}{3-\xi}$.}
  \end{aligned}
\end{equation*}
Since $d(\xi)$ is bounded on $\mathbb{D}'$, Lemma
\ref{lem:RB-calculus} (iii)-(iv) and Theorem \ref{thm:blunk} give the desired assertion.
\qed\end{proof}

\section{L1 scheme}
Now we discuss one time-stepping scheme  of finite difference type for simulating subdiffusion
-- the L1 scheme \cite{LinXu:2007,SunWu:2006} -- which is easy to implement and converges
robustly for nonsmooth data, hence very popular. However, unlike convolution quadrature,
finite difference type methods are generally challenging to analyze.
For the subdiffusion case, i.e., $\alpha\in(0,1)$, it approximates the (Caputo) fractional derivative $\partial_t^\alpha u(t_n)$
with a time step size $\tau$ by
\begin{equation}\label{eqn:L1approx}
   \begin{aligned}
     \partial_t^\alpha u(t_n) &= \frac{1}{\Gamma(1-\alpha)}\sum^{n-1}_{j=0}\int^{t_{j+1}}_{t_j}
        u^\prime(s)(t_n-s)^{-\alpha}\, ds \\
     &\approx \frac{1}{\Gamma(1-\alpha)}\sum^{n-1}_{j=0} \frac{u(t_{j+1})-u(t_j)}{\tau}\int_{t_j}^{t_{j+1}}(t_n-s)^{-\alpha}ds \\
     &=\sum_{j=0}^{n-1}b_j\frac{u(t_{n-j})-u(t_{n-j-1})}{\tau^\alpha}\\
     &=\tau^{-\alpha} [b_0u(t_n)-b_{n-1}u(t_0)+\sum_{j=1}^{n-1}(b_j-b_{j-1})u(t_{n-j})]
     =:\bar \partial_\tau^\alpha u^n.
   \end{aligned}
 \end{equation}
where the weights $b_j$ are given by
\begin{equation}\label{eqn:bj}
b_j=((j+1)^{1-\alpha}-j^{1-\alpha})/\Gamma(2-\alpha),\ j=0,1,\ldots,N-1.
\end{equation}

For the case $\alpha\in(1,2)$, the L1--approximation reads \cite{SunWu:2006}
\begin{align*}
\partial_t^\alpha u(t_{n-\frac12})
&\approx  \frac{\tau^{-\alpha}}{\Gamma(3-\alpha)} \Big[a_0\delta_t u^{n-\frac12}
   - \sum_{j=1}^{n-1}(a_{n-j-1}-a_{n-j})\delta_t u^{j-\frac12}-a_{n-1}\tau \partial_tu(0) \Big] \\
&=:\bar\partial_\tau^\alpha u^{n},
\end{align*}
where $\delta_tu^{j-\frac12}=u^{j}-u^{j-1}$ denotes central difference, and
$a_j=(j+1)^{2-\alpha}-j^{2-\alpha}$, and we have abused the notation
$\bar\partial_\tau^\alpha u^n$ for approximating $\partial_t^\alpha u(t_{n-\frac12})$.
Formally, it can be obtained by applying
\eqref{eqn:L1approx} to the first derivative $\partial_t u$, in view of the identity
$\partial_t^\alpha u = \partial_t^{\alpha-1} (\partial_t u)$, and then discretizing the
$\partial_tu$ with the Crank-Nicolson type method. The scheme requires $\partial_tu(0)$,
in addition to the initial condition $u(0)$.
Accordingly, we approximate the right hand
side of \eqref{eqn:fde} by a Crank-Nicolson type scheme.
In sum, the L1 scheme reads
\begin{equation}\label{eqn:L1}
  \left\{\begin{aligned}
    \bar\partial_\tau^\alpha u^n &= Au^n + f^n, & 0<\alpha <1,\\
    \bar\partial_\tau^{\alpha}u^{n} &= A(u^n+u^{n-1})/2 + (f^n+f^{n-1})/2, & 1<\alpha <2.
  \end{aligned}\right.
\end{equation}
\begin{remark}
For $\alpha\in(0,1)$, Lin and Xu \cite{LinXu:2007} proved that the L1 scheme
is uniformly $O(\tau^{2-\alpha})$ accurate for $C^2$ solutions; and for
$\alpha\in (1,2)$, Sun and Wu \cite{SunWu:2006} showed that it is uniformly
$O(\tau^{3-\alpha})$ accurate for $C^3$ solutions. It is worth noting that even for
smooth initial data and source term, the solution of fractional-order PDEs may not be
$C^2$ in general. In fact, the L1 scheme is generally only first-order
\cite{JinLazarovZhou:2016ima,JinZhou:2016}.
\end{remark}

For the analysis, we recall the polylogarithmic function $\mathrm{Li}_p(z)$, $p\in\mathbb{R}$
and $z\in\mathbb{C}$, defined by
\begin{equation*}
\mathrm{Li}_p(z)=\sum_{j=1}^\infty \frac{z^j}{j^p}.
\end{equation*}
The function $\mathrm{Li}_p(z)$ is well defined on $\{z:|z|<1\}$, and it can be analytically continued
to the split complex plane $\mathbb{C}\setminus [1,\infty)$ \cite{Flajolet:1999}. With $z=1$,
it recovers the Riemann zeta function $\zeta(p)=\mathrm{Li}_p(1)$. First we state the solution representation.

\begin{lemma}\label{lem:L1-repr}
The discrete solution $u(\xi)$ of the L1 scheme \eqref{eqn:L1} satisfies
\begin{equation}\label{eqn:delta_L1}
  (\tau^{-\alpha}\delta(\xi)-A)u(\xi)= f(\xi),
\end{equation}
with the generating functions
\begin{align*}
&\delta(\xi)= \left\{\begin{array}{ll}\displaystyle
  \frac{(1-\xi)^2}{\xi\Gamma(2-\alpha)}\mathrm{Li}_{\alpha-1}(\xi), &\quad\alpha\in(0,1),\\[1.2ex]
  \displaystyle \frac{2(1-\xi)^3}{\xi(1+\xi)\Gamma(3-\alpha)}\mathrm{Li}_{\alpha-2}(\xi), &\quad\alpha\in(1,2),
\end{array}\right. \\[10pt]
&f(\xi)= \left\{\begin{aligned}
\sum_{n=1}^\infty f^n\xi^n,\quad \alpha\in(0,1),\\
\frac{\xi}{1+\xi}\sum_{n=0}^\infty f^n\xi^n + \frac{1}{1+\xi}\sum_{n=1}^\infty f^n\xi^n, \quad\alpha\in(1,2).
\end{aligned}\right.
\end{align*}
\end{lemma}
\begin{proof}$\,\,$
We first show the representation for $\alpha\in(0,1)$,
and the case $\alpha\in(1,2)$ is analogous.
Multiplying both sides of \eqref{eqn:L1} by $\xi^n$ and summing over $n$ yield
\begin{equation*}
  \sum_{n=1}^\infty \bar\partial_\tau^\alpha u^n\xi^n - A  u(\xi) = \sum_{n=1}^\infty f^n\xi^n,
\end{equation*}
upon noting $u^0=0$.
Now we focus on the term $\sum_{n=1}^\infty \bar\partial_\tau^\alpha u^n\xi^n$.
Since $u^0=0$, by the convolution rule \eqref{conv-rule}, we have
\begin{equation*}
\begin{split}
  \sum_{n=1}^\infty \bar\partial_\tau^\alpha u^n\xi^n & = \tau^{-\alpha}\sum_{n=1}^\infty \Big(b_0u^n+\sum_{j=1}^{n-1}(b_j-b_{j-1})u^{n-j}\Big) \xi^n\\
    & = \tau^{-\alpha} \sum_{n=1}^\infty \Big(\sum_{j=0}^{n-1} b_j u^{n-j}\Big)\xi^n
    - \tau^{-\alpha}\sum_{n=1}^\infty \Big(\sum_{j=1}^{n-1} b_{j-1} u^{n-j}\Big)\xi^n \\[5pt]
    & = \tau^{-\alpha}(1-\xi) b(\xi) u (\xi).
\end{split}
\end{equation*}
Using the polylogarithmic function $\mathrm{Li}_p(z)$, $b(\xi)$ is given by
\begin{align*}
  b(\xi) &= \frac{1}{\Gamma(2-\alpha)}\sum_{{j=0}}^\infty (( j+1 )^{1-\alpha}-j^{1-\alpha}) \xi^j\\
    &= \frac{1-\xi}{\xi\Gamma(2-\alpha)} \sum_{j=1}^\infty j^{1-\alpha} \xi^j= \frac{(1-\xi)\mathrm{Li}_{\alpha-1}(\xi)}{\xi\Gamma(2-\alpha)},
\end{align*}
from which the desired solution representation follows directly.
\qed\end{proof}

We shall need the following result, which is of independent interest.
\begin{lemma}\label{lem:L1-angle1}
For $\alpha\in(0,1)$ and $\xi \in \mathbb{D}'$, we have $\psi(\xi) := \frac{(1-\xi)^2}{\xi}\mathrm{Li}_{\alpha-1}(\xi)\in \Sigma_{\frac{\pi\alpha}{2}}$.
\end{lemma}
\begin{proof}$\,\,$
It suffice to consider $\xi=e^{-\mathrm{i}\theta}$ with $\theta\in(0,\pi]$,
since the case $\theta\in(\pi,2\pi)$ can be proved similarly. Using the identity
\begin{equation*}
    \frac{(1-\xi)^2}{\xi} = \frac1{\xi}+\xi-2= e^{-\mathrm{i}\theta}+e^{\mathrm{i}\theta}-2 = 2\cos \theta -2,
\end{equation*}
we have
\begin{equation*}
    \arg( {(1-\xi)^2}/{\xi}) = \arg(2\cos \theta -2) = -\pi.
\end{equation*}
Moreover, we have the expansion \cite[equation (13.1)]{wood1992computation}
\begin{align}\label{Exp-Li-alpha-1}
&\frac{\mathrm{Li}_{\alpha-1}(\xi)}{\Gamma(2-\alpha)} \\
&=(-2\pi\mathrm{i})^{\alpha-2}\sum_{k=0}^\infty\left(k+1-\tfrac{\theta}{2\pi}\right)^{\alpha-2}  +
      (2\pi\mathrm{i})^{\alpha-2}\sum_{k=0}^\infty\left(k+\tfrac{\theta}{2\pi}\right)^{\alpha-2} \nonumber \\
        & =(2\pi)^{\alpha-2}
        \left(\cos((2-\alpha)\tfrac\pi2)(A(\theta)+B(\theta)) - {\mathrm i}\sin((2-\alpha)\tfrac\pi2) (A(\theta)-B(\theta)) \right),\nonumber
\end{align}
where
\begin{equation*}
    A(\theta)=\sum_{k=0}^\infty\left(k+\tfrac{\theta}{2\pi}\right)^{\alpha-2} \quad \text{and}\quad
    B(\theta)=\sum_{k=0}^\infty\left(k+1-\tfrac{\theta}{2\pi}\right)^{\alpha-2}  .
\end{equation*}
Both series converge for $\alpha\in(0,1)$.
Since for $\theta\in(0,\pi]$, $(k+\tfrac{\theta}{2\pi} )^{\alpha-2} > (k+1-\tfrac{\theta}{2\pi} )^{\alpha-2}>0$, there holds
\begin{equation*}
    \frac{A(\theta)-B(\theta)}{A(\theta)+B(\theta)} \in (0,1),
\end{equation*}
and we deduce
\begin{equation*}
    \arg( \mathrm{Li}_{\alpha-1}(\xi) ) \in [-\pi,-\pi+\alpha\pi/2)\quad \text{for}~~ \xi=e^{-\mathrm{i}\theta},~\theta\in(0,\pi].
\end{equation*}
Therefore, we have
\begin{equation*}
    \arg\Big(\frac{(1-\xi)^2}{\xi}\mathrm{Li}_{\alpha-1}(\xi) \Big)=\arg( \mathrm{Li}_{\alpha-1}(\xi) )
    +\arg( {(1-\xi)^2}/{\xi})\in[0,\alpha\pi/2) .
\end{equation*}
This completes the proof of the lemma.
\qed\end{proof}

\begin{lemma}\label{lem:delta-bdd}
For the function $\delta(\xi)$ defined by \eqref{eqn:delta_L1}, there holds
\begin{equation*}
 (1-\xi)(1+\xi)\delta^\prime(\xi) = d(\xi) \delta(\xi)
\end{equation*}
with
\begin{equation*}
    d(\xi)= \left\{\begin{array}{ll}\displaystyle
  (1+\xi)\left(-2+\frac{1-\xi}{\xi}\frac{\mathrm{Li}_{\alpha-2}(\xi)-\mathrm{Li}_{\alpha-1}(\xi)}
  {\mathrm{Li}_{\alpha-1}(\xi)}\right), &\quad\alpha\in(0,1),\\[1.2ex]
  \displaystyle (1+\xi)\left(-3+\frac{1-\xi}{\xi}\frac{\mathrm{Li}_{\alpha-3}(\xi)-\mathrm{Li}_{\alpha-2}(\xi)}
  {\mathrm{Li}_{\alpha-2}(\xi)}\right) + (\xi-1), &\quad\alpha\in(1,2).
\end{array}\right.
\end{equation*}
where $d(\xi)$ is uniformly bounded on ${\mathbb D}'$.
\end{lemma}
\begin{proof}$\,\,$
It suffices to consider the case $\alpha\in(0,1)$, while the other case follows analogously.
Since $\mathrm{Li}_{\alpha-1}(\xi)$ is analytic, by termwise differentiation,
$\mathrm{Li}^\prime_{\alpha-1}(\xi)=\xi^{-1}\mathrm{Li}_{\alpha-2}(\xi)$. Thus, with $c_\alpha=1/\Gamma(2-\alpha)$, we have
\begin{equation*}
    \delta^\prime(\xi) = c_\alpha\Big({-}\frac{2(1-\xi)}{\xi}\mathrm{Li}_{\alpha-1}(\xi) -\frac{(1-\xi)^2}{\xi^2}\mathrm{Li}_{\alpha-1}(\xi)
    +\frac{(1-\xi)^2}{\xi^2}\mathrm{Li}_{\alpha-2}(\xi)\Big),
\end{equation*}
from which the expression of $d(\xi)$ follows.
By using the asymptotic expansion (see \cite[equation (9.3)]{wood1992computation}
or \cite[Theorem 1]{Flajolet:1999})
\begin{align}\label{L1-asymp1}
\mathrm{Li}_{p}(e^{-\mathrm{i}\theta})=\Gamma(1-p)(\mathrm{i}\theta)^{p-1} + o(\theta^{p}) , \quad\mbox{as } \theta\to0,
\end{align}
we deduce
\begin{equation*}
  \lim_{\begin{subarray}{ll}\xi\to 1\\
  \xi\in{\mathbb D}'
  \end{subarray}}\frac{1-\xi}{\xi}\frac{\mathrm{Li}_{\alpha-2}(\xi)-\mathrm{Li}_{\alpha-1}(\xi)}{\mathrm{Li}_{\alpha-1}(\xi)}
  =\frac{\Gamma(3-\alpha)}{\Gamma(2-\alpha)} = 2-\alpha  .
\end{equation*}
Hence, $d(\xi)$ is bounded if
$\xi=e^{-{\mathrm i}\theta}$ is close to $1$. Meanwhile,
if $\xi=e^{-{\mathrm i}\theta}$ and $\theta$ is away from the two end-points of the interval $(0,2\pi)$,  then
\eqref{Exp-Li-alpha-1} implies that $|{\rm Li}_{\alpha-1}(\xi)|$ has a positive lower bound and $|{\rm Li}_{\alpha-2}(\xi)|$ has an upper bound. Hence
$d(\xi)$ is bounded.
\qed\end{proof}

Now we can give the discrete maximal regularity for the L1 scheme \eqref{eqn:L1approx}.
\begin{theorem}\label{thm:L1}
Let $X$ be a UMD space, $0<\alpha<1$ or $1<\alpha<2$, and let $A$ be an $R$-sectorial operator on $X$ of angle $\alpha\pi/2$. Then the L1 scheme \eqref{eqn:L1}
satisfies the following discrete maximal regularity
\begin{align*}
&\|(\bar\partial_\tau^\alpha u^n)_{n=1}^N\|_{\ell^p(X)} + \|(Au^n)_{n=1}^N\|_{\ell^p(X)}
\leq
  \left\{
  \begin{aligned}
c_{p,X}c_R\|(f^n)_{n=1}^N\|_{\ell^p(X)},\;\mbox{if}\; 0<\alpha<1,\\
c_{p,X}c_R\|(f^n)_{n=0}^N\|_{\ell^p(X)},\;\mbox{if}\; 1<\alpha<2,
  \end{aligned}
  \right.
\end{align*}
where the constant $c_{p,X}$ is independent of $N$, $\tau$ and $A$,
and $c_R$ denotes the $R$-bound of the set of operators
$\{zR(z;A) :z\in \Sigma_{\alpha\pi/2}\}$.
\end{theorem}
\begin{proof}$\,\,$
First we consider the case $0<\alpha<1$. Upon setting $f^0=0$, Lemmas \ref{lem:L1-repr} and \ref{lem:L1-angle1} yield
\begin{equation*}
  (\bar\partial_\tau^\alpha u)(\xi) = M(\xi)f(\xi)\quad \mbox{with } M(\xi)= \tau^{-\alpha}\delta(\xi)\left(\tau^{-\alpha}\delta(\xi)-A\right)^{-1},
\end{equation*}
where $\delta(\xi)$ is defined by \eqref{eqn:delta_L1}. By Lemma \ref{lem:L1-angle1}, we have
\begin{equation*}
  \{M(\xi):\xi\in\mathbb{D}^\prime\}\subset
  \{z R(z;A):z\in \Sigma_{\alpha\pi/2}\},
\end{equation*}
where the latter set is $R$-bounded by assumption. Meanwhile,
\begin{equation*}
  (1-\xi)(1+\xi)M^\prime(\xi) = d(\xi)M(\xi)-d(\xi)M(\xi)^2,
\end{equation*}
where, by Lemma \ref{lem:delta-bdd}, $d(\xi)$ is uniformly bounded on $\mathbb{D}^\prime$.
By Lemma \ref{lem:RB-calculus}, the set
$\{(1-\xi)(1+\xi)M^\prime(\xi):\xi\in\mathbb{D}^\prime\}$ is $R$-bounded. Thus we deduce from
Theorem \ref{thm:blunk} the desired assertion.

Next we consider the case of $1<\alpha<2$. In this case, we let $g^0=0$ and $g^n=f^n$, $n\geq1$, to obtain
\begin{equation*}
  (\bar\partial_\tau^\alpha u)(\xi)=  \tfrac12 \xi  M(\xi)f(\xi)
  +  \tfrac12  M(\xi) g(\xi),
\end{equation*}
with $M(\xi)= \tau^{-\alpha}\delta(\xi)(\tau^{-\alpha}\delta(\xi)-A)^{-1}$.
In view of the relation $\delta(\xi)=\frac{2}{\Gamma(3-\alpha)}\frac{1-\xi}{1+\xi}\psi(\xi)$,
by Lemma \ref{lem:L1-angle1} and since the function $(1-\xi)/(1+\xi)$ maps $\mathbb{D}^\prime$
into the imaginary axis, we deduce
\begin{equation*}
  \{M(\xi):\xi\in\mathbb{D}^\prime\}\subset \{\lambda R(\lambda;A):\lambda\in \Sigma_{\alpha\pi/2}\}.
\end{equation*}
The rest of the proof follows like before, using Lemma \ref{lem:delta-bdd}.
\qed\end{proof}

\begin{remark}
For the model \eqref{eqn:fde} with $\alpha\in(0,1)$,
the piecewise constant discontinuous Galerkin (PCDG) method
 in \cite{McLeanMustapha:2015} leads to a time-stepping scheme  identical to the L1 scheme.
The PCDG is given by: find $u^n$ such that
\begin{equation*}
  \int_{t_{n-1}}^{t_n}\partial_t^{\alpha}u(s) ds = \int_{t_{n-1}}^{t_n}Au^n(s)ds + \int_{t_{n-1}}^{t_n}f(s)ds,\quad n=1,2,\ldots,N.
\end{equation*}
By letting $f^n=\tau^{-1}\int_{t_{n-1}}^{t_n}f(s)ds$, we obtain
\begin{equation*}
  \tau^{-1}\int_{t_{n-1}}^{t_n}\partial_t^{\alpha}u(s)ds = Au^n + f^n,\quad n=1,\ldots,N.
\end{equation*}
Next we derive the explicit expression for the discrete approximation $\bar{\partial}_\tau^\alpha u^n$
\begin{equation*}
  \bar\partial_\tau^\alpha u^n = \tau^{-1}\int_{t_{n-1}}^{t_n}\partial_t^\alpha u(s)ds =\tau^{-\alpha}\sum_{j=1}^n\beta_{n-j}u^j,
\end{equation*}
where $\beta_0=1$ and $\beta_j=(j+1)^{1-\alpha}-2j^{1-\alpha}+(j-1)^{1-\alpha}$, $j=1,2,\ldots$.
With the weights $b_j$ in \eqref{eqn:bj}, we have $\beta_j=b_j-b_{j-1}$, for $j=1,2,\ldots$, and $\beta_0=b_0$.
Hence, the PCDG approximation $\bar\partial_\tau^\alpha u^n$ reads
\begin{equation*}
  \bar{\partial}_\tau^\alpha u^n = \tau^{-\alpha}b_0u^n + \tau^{-\alpha}\sum_{j=1}^{n-1}(b_j-b_{j-1})u^{n-j}.
\end{equation*}
Thus it is identical with the L1 scheme, and Theorem \ref{thm:L1} applies.
\end{remark}

\section{Explicit Euler method}\label{sec:Explicit-Euler}
Now we analyze the explicit Euler method for discretizing \eqref{eqn:fde} in time:
\begin{equation}\label{eqn:Explicit-Euler}
  \bar\partial_\tau^\alpha u^n = Au^{n-1}+f^{n-1},
  \quad n\ge 1 ,
\end{equation}
where the approximation $\bar\partial_\tau^\alpha u^n$ denotes the BE approximation \eqref{eqn:BE}.
A variant of the scheme was analyzed in \cite{YusteAcedo:2005}.
By multiplying \eqref{eqn:Explicit-Euler} by $\xi^n$ and summing up the results
for $n=1,2,\dots,$ we obtain
\begin{equation*}
  (\tau^{-\alpha}\delta(\xi)-A)u(\xi) = f(\xi)\quad \mbox{and}\quad
(\bar\partial_\tau^\alpha u)(\xi) =
 \tau^{-\alpha}\xi \delta(\xi)u(\xi),
\end{equation*}
with
\begin{equation*}
  \delta(\xi) =\tfrac{(1-\xi)^\alpha}{\xi} .
\end{equation*}

Recall that the numerical range $S(A)$ of an operator $A$ is defined by \cite[pp. 12]{Pazy:1983}
\begin{equation*}
  S(A) = \{\langle x^*,Ax\rangle: x\in X,x^*\in X^*, \|x\|_X=\|x^*\|_{X^*}=\langle x^*,x\rangle =1\}.
\end{equation*}
We denote by $r(A)=\sup_{z\in S(A)}|z|$ the radius of the numerical range $S(A)$, known as numerical radius. Recall that
\cite[Theorem 3.9, Chapter 1, pp. 12]{Pazy:1983}
\begin{equation}\label{eqn:Resolvent}
  \|R(z;A)\|_{\mathcal{B}(X)}\leq \mathrm{dist}(z,\overline{S(A)})^{-1} ,\quad \forall\, z\in \mathbb{C}\setminus \overline{S(A)},
\end{equation}
where $\overline{S(A)}$ denotes of the closure of $S(A)$ in $\mathbb{C}$, and $\mathrm{dist}
(z,\overline{S(A)})$ is the distance of $z$ from $\overline{S(A)}$.

The next theorem gives the maximal $\ell^p$-regularity of the explicit Euler method \eqref{eqn:Explicit-Euler},
if $\tau^\alpha r(A)$ is smaller than some given positive constant.

\begin{theorem}\label{THM:Explicit}
Let $X$ be a UMD space, $0<\alpha<1$ or $1<\alpha<2$, and let $A$ be an $R$-sectorial operator
of angle $\alpha\pi/2$ such that $S(A)\subset {\mathbb C}\backslash\Sigma_{\varphi}$
for some $\varphi\in(\alpha\pi/2,\pi]$.
Then, under the condition
\begin{align}\label{condition:rA}
 \tau^\alpha r(A)\le
 2^\alpha\bigg[\sin\bigg(\frac{\varphi-\alpha\pi/2}{2-\alpha}\bigg)\bigg]^\alpha
 -\epsilon  ,
\end{align}
the scheme \eqref{eqn:Explicit-Euler} satisfies the following discrete maximal regularity
\begin{equation*}
\|(\bar\partial_\tau^\alpha u^n)_{n=1}^N\|_{\ell^p(X)}
+\|(Au^n)_{n=1}^{N-1}\|_{\ell^p(X)}\leq c_{p,X}(1+c_R)\|(f^n)_{n=0}^{N-1}\|_{\ell^p(X)},
\end{equation*}
where the constant $c_{p,X}$ depends only on
$\epsilon$, $\varphi$ and $\alpha$
{\rm(}independent of $\tau$ and $A${\rm)}, and
$c_R$ denotes the $R$-bound of the set
$\{zR(z;A) :z\in \Sigma_{\alpha\pi/2}\}$.
\end{theorem}
\begin{proof}$\,\,$
For $\xi=e^{\mathrm{i}\theta}$, $\theta\in(0,2\pi)$, we have
\begin{align*}
\frac{\delta(e^{\mathrm{i}\theta})}{\tau^\alpha}=\frac{2^\alpha[\sin(\theta/2)]^\alpha}{\tau^\alpha}
e^{\mathrm{i}[-\alpha\pi/2-(1-\alpha/2)\theta]} ,
\end{align*}
which is a parametric curve contained in the sector
${\mathbb C}\backslash\overline\Sigma_{\alpha\pi/2}$.
Let $\Gamma =\{\tau^{-\alpha}\delta(e^{\mathrm{i}\theta}):\theta\in(0,2\pi)\}$. It suffices to prove that
the family of operators $\{zR(z;A):z\in \Gamma\}$ is $R$-bounded.
Since $\{zR(z;A):z\in \Sigma_{\alpha\pi/2}\}$ is $R$-bounded,
by Lemma \ref{lem:Blunk}, we have $\{zR(z;A):z\in \Gamma\cap \Sigma_{\phi}\}$ is $R$-bounded
for some $\phi\in(\alpha\pi/2,\varphi]$, where
$\phi$ depends on $c_R$ and $\alpha$.
It remains to prove that
$\{zR(z;A):z\in \Gamma \backslash\Sigma_{\phi}\}$ is also $R$-bounded.
Note that ${\rm arg}(\tau^{-\alpha}\delta(e^{\mathrm{i}\theta}))\in {\mathbb C}\backslash\Sigma_{\varphi}$ is equivalent to
\begin{equation}\label{eqn:theta-range}
\frac{\varphi-\alpha\pi/2}{1-\alpha/2}<\theta<
2\pi-\frac{\varphi-\alpha\pi/2}{1-\alpha/2}  .
\end{equation}
Meanwhile, since for $\theta\in(0,\pi)$, $|\delta(e^{\mathrm{i}\theta})|=2^\alpha[\sin(\theta/2)]^\alpha$
is strictly monotonically increasing in $\theta$, for such $\theta$ satisfying \eqref{eqn:theta-range}, there holds
$$
\bigg|\frac{\delta(e^{\mathrm{i}\theta})}{\tau^\alpha}\bigg|
\ge \frac{2^\alpha\big[\sin\big(\frac{\varphi-\alpha\pi/2}{2-\alpha}\big)\big]^\alpha}{\tau^\alpha}  .
$$
If \eqref{condition:rA} is satisfied, then
\begin{align}\label{dist-Gamma-SA}
 (1-\epsilon_{\alpha,\varphi})|\delta(e^{\mathrm{i}\theta}) |\ge \tau^\alpha r(A)
\end{align}
for some $\epsilon_{\alpha,\varphi}>0$.
Now consider the curve $\Gamma_0:=\{\delta(e^{\mathrm{i}\theta}): \theta\in(0,2\pi)\}$ and
the closed region $D_0:=\{s\Gamma_0:s\in[0,1]\}$,
which are fixed (and independent of $\tau$).
Since $\overline{S(A)}\subset{\mathbb C}\backslash\Sigma_\varphi$,
it follows from \eqref{dist-Gamma-SA} that
\begin{equation*}
{\rm dist}(z,\tau^{\alpha}\overline{S(A)})\ge
{\rm dist}(\Gamma_0\backslash\Sigma_{\phi}, (1-\epsilon_{\alpha,\varphi})D_0\backslash\Sigma_\varphi)
\ge C^{-1}
\quad \mbox{ for }z\in \Gamma_0\backslash\Sigma_{\phi} .
\end{equation*}
where the constant $C$ depends on the parameters
$\epsilon$, $\alpha$, $\varphi$ and $\phi$, but is independent of $\tau$
(since both $\Gamma_0\backslash\Sigma_{\phi}$ and
$(1-\epsilon_{\alpha,\varphi})D_0\backslash\Sigma_\varphi$
are fixed closed subsets of ${\mathbb C}$, independent of $\tau$).
Since $\Gamma=\tau^{-\alpha}\Gamma_0$,
the last inequality yields (via scaling)
\begin{equation*}
{\rm dist}(z,\overline{S(A)})\ge \tau^{-\alpha} C^{-1}
 \quad\mbox{ for }z\in \Gamma\backslash\Sigma_{\phi}.
\end{equation*}
Hence there exists a finite number of balls $B(z_j,\rho)$ of radius
$\rho=\frac{1}{4}\tau^{-\alpha} C^{-1}$, $z_j\in\Gamma$, which can cover
$\Gamma \backslash\Sigma_{\phi}$, and further, the number of balls is bounded by a constant which depends
only the parameters $\epsilon$, $\alpha$, $\varphi$ and $\phi$, independent of $\tau$ and $A$.
For each ball $B(z_j,\rho)$, $\{zR(z;A):z\in B(z_j,\rho)\}$ is $R$-bounded
and its $R$-bound is at most (see Lemma \ref{R-bound-analytic} below)
$$
\sup_{z\in B(z_j,\rho)} 2|z|\|R(z;A)\|_{\mathcal{B}(X)}
\le \sup_{z\in B(z_j,\rho)} 2|z|\, {\rm dist}(z,\overline{S(A)})^{-1}
\le C ,
$$
where we have used the estimate \eqref{eqn:Resolvent} in the first inequality.
Then Lemma \ref{lem:RB-calculus} (iii) implies that
$\{zR(z;A):z\in \Gamma \backslash\Sigma_{\phi}\}$ is also $R$-bounded.
\qed\end{proof}

\begin{remark}
The constant in condition \eqref{condition:rA} is sharp. The scaling factor $\tau^\alpha$
is one notable feature of the model \eqref{eqn:fde}, and for $\alpha\in(0,1)$, the
exponent $\alpha$ agrees with that in the stability condition in \cite{YusteAcedo:2005}.
Hence, the smaller the fractional order $\alpha$ is, the smaller the step size $\tau$
should be taken.
\end{remark}

\begin{remark}
The condition \eqref{condition:rA} covers bounded operators, e.g., finite element approximations
of a self-adjoint second-order elliptic operator. For a self-adjoint discrete approximation,
the numerical range $S(A)$ is the closed interval spanned by the largest and smallest eigenvalues,
but in general, the numerical range $S(A)$ has to be approximated \cite[Section 5.6]{GustafsonRao:1997}.
\end{remark}

\begin{lemma}[{\bf $R$-boundedness of operator-valued analytic functions}]
\label{R-bound-analytic}
If the function $F:\overline B(z_0,\rho)\rightarrow \mathcal{B}(X)$ is
analytic in a neighborhood of the ball $\overline B(z_0,\rho)$, centered at $z_0$ with radius $\rho$,
then the set
of operators $\{F(z):\lambda\in B(z_0,\rho/2)\}$ is $R$-bounded on $X$, and its
$R$-bound is at most
$$
2\sup_{z\in B(z_0,\rho)}\|F(z)\|_{\mathcal{B}(X)} .
$$
\end{lemma}
\begin{proof}$\,\,$
The analyticity implies the existence of a power series expansion
\begin{align*}
F(z)=\sum_{n=0}^\infty \frac{F_n}{n!}(z-z_0)^n
\end{align*}
where $F_n$, $n=0,1,2,\dots$ are bounded linear operators on $X$
and the series converges absolutely in $B(z_0,\rho)$.
Moreover, by Cauchy's integral formula,
\begin{align*}
\|F_n\|_{\mathcal{B}(X)}=
\bigg\|\frac{1}{2\pi\mathrm{i}}\int_{\partial B(z_0,\rho)}\frac{n! F(z)}{(z-z_0)^{n+1}} dz\bigg\|_{\mathcal{B}(X)}
\le \rho^{-n}n!\sup_{z\in B(z_0,\rho)}\|F(z)\|_{\mathcal{B}(X)} .
\end{align*}
Hence, for $z_j\in B(z_0,\rho/2)$ and $u_j\in X$, $j=1,2,\dots,m$,
Minkowski's inequality implies
\begin{align*}
\int_0^1\bigg\|\sum_{j=1}^m r_j(s)F(z_j)u_j\bigg\|_X ds
&\leq \sum_{n=0}^\infty \frac{(\rho/2)^n}{n!}
\int_0^1\bigg\|\sum_{j=1}^mr_j(s)\bigg(\frac{z_j-z_0}{\rho/2}\bigg)^n F_nu_j\bigg\|_X ds \\
&\leq 2\sum_{n=0}^\infty \frac{(\rho/2)^n}{n!}
\int_0^1\bigg\| \sum_{j=1}^mr_j(s)F_nu_j\bigg\|_X ds \\
&\leq 2\sum_{n=0}^\infty \frac{(\rho/2)^n\|F_n\|_{\mathcal{B}(X)} }{n!}
\int_0^1\bigg\| \sum_{j=1}^mr_j(s)u_j\bigg\|_X ds \\
&\leq 2\sum_{n=0}^\infty  2^{-n}\sup_{z\in B(z_0,\rho)}\|F(z)\|_{\mathcal{B}(X)}
\int_0^1\bigg\| \sum_{j=1}^mr_j(s)u_j\bigg\|_X ds \\
&\leq 2 \sup_{z\in B(z_0,\rho)}\|F(z)\|_{\mathcal{B}(X)}
\int_0^1\bigg\| \sum_{j=1}^mr_j(s)u_j\bigg\|_X ds ,
\end{align*}
where the second line follows from \cite[Proposition 2.5]{KunstmannWeis:2004}.
This shows that the family of operators  $\{F(z):z\in B(z_0,\rho/2)\}$ is $R$-bounded.
\qed\end{proof}

\section{Fractional Crank--Nicolson method}\label{sec:CN}
By the fractional Crank--Nicolson method, we mean the following scheme:
\begin{equation}\label{eqn:CN}
\bar\partial_\tau^\alpha u^n = (1-\tfrac{\alpha}{2})Au^{n}+\tfrac{\alpha}{2}Au^{n-1}+(1-\tfrac{\alpha}{2})f^{n}+\tfrac{\alpha}{2}f^{n-1},
\end{equation}
where the approximation $\bar\partial_\tau^\alpha u^n$ denotes the BE approximation \eqref{eqn:BE}.
When $\alpha=1$, \eqref{eqn:CN} coincides with the standard Crank--Nicolson method.
For any $0<\alpha<2$, one can verify that it is second-order
in time, provided that the solution is sufficiently smooth \cite{JinLiZhou:2016CN}.
By multiplying \eqref{eqn:Explicit-Euler} by $\xi^n$ and summing up the results
for $n=1,2,\dots,$ we obtain
\begin{align*}
&(\tau^{-\alpha}\delta(\xi)-A)u(\xi) = f(\xi) \\
&(\bar\partial_\tau^\alpha u)(\xi) =
\left(1-\tfrac{\alpha}{2}+\tfrac{\alpha}{2}\xi\right)\tau^{-\alpha}\delta(\xi)u(\xi),
\end{align*}
with
\begin{equation*}
  \delta(\xi) =\frac{(1-\xi)^\alpha}{1-\frac{\alpha}{2} +\frac{\alpha}{2}\xi} .
\end{equation*}

First, we prove the  maximal $\ell^p$-regularity for
\eqref{eqn:CN} in the case $0<\alpha<1$.

\begin{theorem}\label{THM:CN-1}
Let $X$ be a UMD space, $0<\alpha<1$, and let
$A$ be an $R$-sectorial operator on $X$ of angle $\alpha\pi/2$ such that $S(A)\subset {\mathbb C}\backslash\Sigma_{\varphi}$
for some $\varphi\in(\alpha\pi/2,\pi]$.
Then the scheme \eqref{eqn:CN} satisfies the following discrete maximal regularity
\begin{equation*}
\|(\bar\partial_\tau^\alpha u^n)_{n=1}^N\|_{\ell^p(X)}
+\|(Au^n)_{n=1}^N\|_{\ell^p(X)}\leq c_{p,X}c_R\|(f^n)_{n=0}^{N}\|_{\ell^p(X)},
\end{equation*}
where the constant $c_{p,X}$ depends only on $\alpha$
{\rm(}independent of $\tau$ and $A${\rm)}, and $c_R$ denotes the $R$-bound of the set of operators
$\{zR(z;A) :z\in \Sigma_{\alpha\pi/2}\}$.
\end{theorem}
\begin{proof}$\,\,$
It suffices to prove that the family of operators
$\big\{\tau^{-\alpha}\delta(\xi)(\tau^{-\alpha}\delta(\xi)-A)^{-1}:\xi\in{\mathbb D}'\big\}$ is $R$-bounded.
In fact, for $\xi=e^{\mathrm{i}\theta}$, $\theta\in(0,2\pi)$, we have
\begin{align*}
\frac{\delta(e^{\mathrm{i}\theta})}{\tau^\alpha }=\frac{2^\alpha[\sin(\theta/2)]^\alpha}{\tau^\alpha \rho(\theta)}
e^{\mathrm{i}(-\frac{\alpha}{2}\pi  +\frac{\alpha}{2}\theta-\psi(\theta))} ,
\end{align*}
where the functions $\rho(\theta)$ and $\psi(\theta)$ are defined respectively by
\begin{align}
&\rho(\theta)
:=\sqrt{(1-\tfrac{\alpha}{2})^2+\tfrac{\alpha^2}{4}+\alpha(1-\tfrac{\alpha}{2})\cos\theta} ,
\end{align}
and
\begin{align}
&\psi(\theta)
:=\arg\Big(1-\frac{\alpha}{2}+\frac{\alpha}{2}\cos\theta+\mathrm{i} \frac{\alpha}{2}\sin\theta\Big)
=\arctan\frac{\frac{\alpha}{2}\sin\theta}{1-\frac{\alpha}{2}+\frac{\alpha}{2}\cos\theta}.\label{eqn:Psi}
\end{align}
It is straightforward to compute
$$
\frac{\alpha}{2}-\psi'(\theta)
=\frac{\frac{\alpha}{2}(1-\alpha)(1-\frac{\alpha}{2})(1-\cos\theta)}{(1-\frac{\alpha}{2}+\frac{\alpha}{2}\cos\theta)^2+\frac{\alpha^2}{4}\sin^2\theta}
\ge 0 .
$$
Thus $\frac{\alpha}{2}\theta-\psi(\theta)$ is an increasing function of $\theta$, taking values from $0$ to $\alpha \pi$ as $\theta$ changes from
$0$ to $2\pi$. Thus $\tau^{-\alpha}\delta(e^{\mathrm{i}\theta})\in\Sigma_{\alpha\pi/2}$,
and by Lemma \ref{lem:RB-calculus}, the set
$\{(1-\frac\alpha2+\frac\alpha2 \xi)\tau^{-\alpha}\delta(\xi)(\tau^{-\alpha}\delta(\xi)-A)^{-1}:\xi\in{\mathbb D}'\}$ is $R$-bounded.
\qed\end{proof}

Let the function $\psi$ be defined in \eqref{eqn:Psi}, and
$\theta_\varphi\in(0,\pi)$ be the unique root of the equation
\begin{align}\label{Eq:theta_varphi}
\psi(\theta_\varphi)-\frac{\alpha}{2}\theta_\varphi
=\varphi-\frac{\alpha\pi }{2} .
\end{align}
Then we have the following result for the case $1<\alpha<2$.

\begin{theorem}\label{THM:CN-2}
Let $X$ be a UMD space, $1<\alpha<2$,
and let $A$ be an $R$-sectorial operator on $X$ of angle $\alpha\pi/2$ such that $S(A)\subset {\mathbb C}\backslash\Sigma_{\varphi}$
for some $\varphi\in(\alpha\pi/2,\pi)$.
Then, under the condition
\begin{align}\label{condition:rA-CN}
 \tau^\alpha r(A)\le
 \frac{2^\alpha[\sin(\theta_\varphi/2)]^\alpha}{  \rho(\theta_\varphi)} -\epsilon  ,
\end{align}
the scheme \eqref{eqn:CN} satisfies the following discrete maximal regularity
\begin{equation*}
\|(\bar\partial_\tau^\alpha u^n)_{n=1}^N\|_{\ell^p(X)}
+\|(Au^n)_{n=1}^N\|_{\ell^p(X)}\leq c_{p,X}(1+c_R)\|(f^n)_{n=0}^{N}\|_{\ell^p(X)} ,
\end{equation*}
where the constant $c_{p,X}$ depends only on
$\epsilon$, $\varphi$ and $\alpha$
{\rm(}independent of $\tau$ and $A${\rm)}, and
$c_R$ denotes the $R$-bound of the set
$\{zR(z;A) :z\in \Sigma_{\alpha\pi/2}\}$.
\end{theorem}

\begin{proof}$\,\,$
If $1<\alpha<2$, then
$$
\frac{\alpha}{2}-\psi'(\theta)
=\frac{\frac{\alpha}{2}(1-\alpha)(1-\frac{\alpha}{2})(1-\cos\theta)}{(1-\frac{\alpha}{2}+\frac{\alpha}{2}\cos\theta)^2+\frac{\alpha^2}{4}\sin^2\theta}
\le 0 .
$$
Hence, $\frac{\alpha}{2}\theta-\psi(\theta)$ is a decreasing function of $\theta$, taking values
from $0$ to $\alpha \pi-2\pi$ as $\theta$ changes from $0$ to $2\pi$. Thus
$\tau^{-\alpha}\delta(e^{\mathrm{i}\theta})\in{\mathbb C}\backslash\Sigma_{\alpha\pi/2}$.
With $\Gamma =\{\tau^{-\alpha} \delta(e^{\mathrm{i}\theta}): \theta\in(0,2\pi)\}$,
it suffices to show that $\{zR(z;A):z\in \Gamma\}$ is $R$-bounded.
Since $\{zR(z;A):z\in \Sigma_{\alpha\pi/2}\}$ is $R$-bounded, by Lemma \ref{lem:Blunk},
$\{zR(z;A):z\in \Gamma\cap \Sigma_{\phi}\}$ is $R$-bounded for some $\phi\in(\alpha\pi/2,\pi)$, where
$\phi$ depends on $c_R$ and $\alpha$. It remains to prove that
$\{zR(z;A):z\in \Gamma \backslash\Sigma_{\phi}\}$ is also $R$-bounded.
However, ${\rm arg}(\tau^{-\alpha}\delta(e^{\mathrm{i}\theta}))\in {\mathbb C}\backslash\Sigma_{\varphi}$ is equivalent to
\begin{equation}\label{eqn:theta-range2}
\theta_\varphi <\theta< 2\pi-\theta_\varphi  ,
\end{equation}
where $\theta_\varphi$ is the unique root of equation \eqref{Eq:theta_varphi}.
Meanwhile, for $ \theta\in(0,\pi)$, $|\delta(e^{\mathrm{i}\theta})|=2^\alpha[\sin(\theta/2)]^\alpha/\rho(\theta)=2^\alpha \sin(\theta/2)^{\alpha-1}\cdot\sin(\theta/2)/\rho(\theta)$ is
monotonically increasing. Hence, for any $\theta$ satisfying \eqref{eqn:theta-range2}, we have
$$
\bigg|\frac{\delta(e^{\mathrm{i}\theta})}{\tau^\alpha}\bigg|
\ge \frac{2^\alpha[\sin(\theta_\varphi/2)]^\alpha}{ \rho(\theta_\varphi)\tau^\alpha }   .
$$
If \eqref{condition:rA-CN} is satisfied then for some positive  constant $\epsilon_{\alpha,\varphi}$,
$$
 (1-\epsilon_{\alpha,\varphi})\bigg|\frac{\delta(e^{\mathrm{i}\theta})}{\tau^\alpha}\bigg|\ge r(A).
$$
By repeating the argument in Theorem \ref{THM:Explicit},
we deduce ${\rm dist}(z,\overline{S(A)})\ge \tau^{-\alpha} C^{-1}$ for $z\in \Gamma\backslash\Sigma_{\phi}$,
where $C$ is some constant which may depend on
 $\epsilon$, $\alpha$, $\varphi$ and $\phi$, but is independent of $\tau$.
Hence, there exists a finite number of balls $B(z_j,\rho)$ of radius
$\rho=\frac{1}{4}\tau^{-\alpha} C^{-1}$, $z_j\in \Gamma$, which can cover
$\Gamma \backslash\Sigma_{\phi}$, and the number of balls is bounded by a constant which depends only
on $\epsilon$, $\alpha$, $\varphi$ and $\phi$, independent of $\tau$ and $A$.
By Lemma \ref{R-bound-analytic}, for each ball $B(z_j,\rho)$, $\{zR(z;A):z\in B(z_j,\rho)\}$
is $R$-bounded and its $R$-bound is at most
$$
2\sup_{z\in B(z_j,\rho)}|z|\|R(z;A)\|_{\mathcal{B}(X)}
\le 2\sup_{z\in B(z_j,\rho)}|z|\, {\rm dist}(z,\overline{S(A)})^{-1}
\le C .
$$
Then Lemma \ref{lem:RB-calculus} (iii) implies that
$\{zR(z;A):z\in \Gamma \backslash\Sigma_{\phi}\}$ is also $R$-bounded.
\qed\end{proof}

\section{Inhomogeneous initial condition}\label{sec:inhomo}
In this section, we consider maximal $\ell^p$-regularity for the problem
\begin{equation}\label{eqn:inh}
    \partial_t^\alpha u(t) = A u(t) ,  \quad t>0
\end{equation}
with nontrivial initial conditions:
\begin{align} \label{eqn:inh}
\begin{aligned}
&u(0)=v,\qquad &&\quad \mbox{(for $0<\alpha<1$)},\\
&u(0)=v ,\,\,\, \partial_tu(0)= w ,&&\quad \mbox{(for $1<\alpha<2$)}.
\end{aligned}
\end{align}
We focus on the BE scheme since other schemes can be analyzed similarly.
For \eqref{eqn:inh}, the BE scheme reads \cite{JinLazarovZhou:2016sisc,JinLiZhou:2016CN}: with $u^0=v $, find $u^n$ such that
\begin{align}\label{Euler-inh}
\begin{aligned}
   \bar\partial_\tau^\alpha (u-v)^n & = A u^n ,\quad n=1,2,\dots \quad \mbox{(for $0<\alpha<1$)},\\
   \bar\partial_\tau^\alpha (u-v-t w)^n &= A u^n ,\quad n=1,2,\dots \quad \mbox{(for $1<\alpha<2$)},
\end{aligned}
\end{align}
where $\bar\partial_\tau^\alpha $ denotes the BE convolution quadrature \eqref{eqn:BE}.

We shall need the scaled $L^p$-norm and weak $L^p$-norm (cf. \cite[section 1.3]{BL1976})
\begin{align}\label{Lpnorm}
&\|(u^n)_{n=1}^N\|_{L^p(X)}
:=\bigg(\tau\sum_{n=1}^N\|u^n\|_X^p\bigg)^{\frac{1}{p}} ,\\
&\|(u^n)_{n=1}^N\|_{L^{p,\infty}(X)}
:=\sup_{\lambda>0} \lambda  |\{n\ge 1: \|u^n\|_X>\lambda \}|^{\frac{1}{p}}\tau^{\frac{1}{p}}.
\end{align}

The main result of this section is the following theorem.

\begin{theorem}\label{thm:Euler-inh}
Let $X$ be a Banach space, $0<\alpha<1$, and let $A$ be a sectorial operator on $X$ of angle $\alpha\pi/2$.
Then the BE scheme
\eqref{Euler-inh} has the following maximal $\ell^p$-regularity
\begin{align*}
&\|(\bar\partial_\tau^\alpha u^n)_{n=1}^N\|_{L^p(X)}+ \|(Au^n)_{n=1}^N\|_{L^p(X)}
\le c_{p}\|v\|_{(X,D(A))_{1-\frac{1}{p\alpha},p}},
&& p\in(1/\alpha,\infty], 
\\
&\|(\bar\partial_\tau^\alpha u^n)_{n=1}^N\|_{L^{p,\infty}(X)}+ \|(Au^n)_{n=1}^N\|_{L^{p,\infty}(X)} \le c_{p}\|v\|_{X},
&& p=1/\alpha,  
\\
&\|(\bar\partial_\tau^\alpha u^n)_{n=1}^N\|_{L^{p}(X)}
+ \|(Au^n)_{n=1}^N\|_{L^{p}(X)}
\le c_{p}\|v\|_{X},
&&  p\in[1,1/\alpha) , 
\end{align*}
where the constant $c_{p}$ depends on the bound of the set of operators
$\{zR(z;A) :z\in \Sigma_{\alpha\pi/2}\}$,
independent of $N$, $\tau$ and $A$.
\end{theorem}

\begin{proof}$\,\,$
By multiplying both sides of \eqref{Euler-inh} by $\xi^n$ and summing over $n$, we have
\begin{equation*}
  \sum_{n=1}^\infty \xi^n \bar\partial_\tau^\alpha (u-v)^n - \sum_{n=1}^\infty Au^n\xi^n = 0 .
\end{equation*}
Let $u(\xi)=\sum_{n=1}^\infty u^n\xi^n$. Then by repeating
the argument in the proof of Theorem \ref{thm:Euler}, we have (with $\delta(\xi)=1-\xi$)
\begin{equation*}
Au(\xi) = A(\tau^{-\alpha}\delta(\xi)^\alpha-A) ^{-1} \tau^{-\alpha}\delta(\xi)^\alpha\frac{\xi}{1-\xi} v ,
\end{equation*}
where the right-hand side is an analytic function in the unit disk.
For $\rho\in(0,1)$, the Cauchy's integral formula and the change of variable $\xi=e^{-\tau z}$ yield
\begin{align}\label{Aun}
Au^n
&=\frac{1}{2\pi\mathrm{i}}\int_{|\xi|=\rho} Au(\xi) \xi^{-n-1} d\xi
=\frac{\tau}{2\pi\rm i}\int_{\Gamma_\rho^\tau} Au(e^{-\tau z}) e^{t_{n} z} d z
=\frac{\tau}{2\pi\rm i}\int_{\Gamma_\rho^\tau} K(z)v  d z ,\nonumber
\end{align}
where the kernel function $K(z)$ is defined by
\begin{equation*}
  K(z)=  e^{t_{n} z} A(\tau^{-\alpha}\delta(e^{-\tau z})^\alpha-A) ^{-1}\tau^{-\alpha}\delta(e^{-\tau z})^\alpha\frac{e^{-\tau z}}{1-e^{-\tau z}},
\end{equation*}
and $\Gamma_\rho^\tau=\{ a+{\rm i} y: y\in(-\pi/\tau,\pi/\tau)\}$ with $a=\tau^{-1}\ln\frac{1}{\rho}>0$.
Since $zR(z;A)$ is bounded for $z\in \Sigma_{\alpha\pi/2}$, $zR(z;A)$ is also bounded for $z\in \Sigma_{\alpha\pi/2+\varepsilon}$ (the angle can be slightly self-improved (cf. \cite[Theorem 5.2 (c)]{Pazy:1983}). Then a standard perturbation argument shows that
there exists $\theta_\varepsilon>0$ (depending on $\varepsilon$) such that
$\delta(e^{-\tau z})^\alpha\in \Sigma_{\alpha\pi/2+\varepsilon}$ when $z\in \Sigma_{\frac{\pi}{2}+\theta_\varepsilon}$.
Let
\begin{align*}
&\Gamma_{\theta_\varepsilon,\kappa}^\tau=
\Big\{\rho e^{{\rm i}\theta_\varepsilon}: \kappa\le \rho\le \frac{\pi}{\tau\sin\theta_\varepsilon}\Big\}
\bigcup \Big\{\kappa e^{{\rm i}\varphi}: -\theta_\varepsilon \le \varphi\le  \theta_\varepsilon\Big\} ,\\
& \Gamma_{\pm}^\tau =\Big\{x \pm {\rm i} \pi/\tau:
\frac{\pi\cos\theta_\varepsilon}{\tau\sin\theta_\varepsilon}
<x< \tau^{-1}\ln\frac{1}{\rho} \Big\} ,
\end{align*}
where $\Gamma_{\theta_\varepsilon,\kappa}^\tau$ is oriented upwards
and $\Gamma_{\theta_\varepsilon,\kappa}^\tau$ is oriented rightwards,
and $0<\kappa<\tau^{-1}\ln\frac{1}{\rho}$.
Then the function
$
K(z) v
$
is analytic in $z$ in the region enclosed by $\Gamma_{\theta_\varepsilon,\kappa}^\tau
$, $\Gamma_{\pm}^\tau $ and $\Gamma_\rho^\tau$. Since the integrals on
$\Gamma_{+}^\tau$ and $\Gamma_{-}^\tau$ cancel each other due to the $2\pi {\rm i}$-periodicity
of the integrand, the Cauchy's theorem yields
\begin{align*}
Au^n = \frac{\tau}{2\pi\rm i}\int_{\Gamma_\rho^\tau}  K(z)v d z = \frac{\tau}{2\pi\rm i} \int_{\Gamma_{\theta_\varepsilon,\kappa}^\tau}
K(z)v d z .
\end{align*}
Then by choosing $\kappa=t_n^{-1}$ in the contour $\Gamma_{\theta_\varepsilon,\kappa}^\tau$, we deduce
\begin{align*}
\|Au^n\|_X
&\leq
c\Big(\int_{\kappa t_n}^{\frac{\pi}{\tau\sin\theta_\varepsilon}}  s^{-1}
e^{s \cos\theta_\varepsilon} d s
+\int_{-\theta_\varepsilon}^{\theta_\varepsilon}
e^{t_{n} \kappa\cos\varphi}d \varphi\Big)\|Av\|_X
\le c\|Av\|_X.
\end{align*}
Similarly, one can show
\begin{align*}
\|u^n\|_X\le c\|v\|_X  \quad \mbox{and}\quad \|Au^n\|_X\le ct_n^{-\alpha} \|v\|_X.
\end{align*}
This last estimate immediately implies
the third assertion of Theorem \ref{thm:Euler-inh}.
Now for $p\in(1/\alpha,\infty]$, we define $E_\tau:X\rightarrow L^\infty({\mathbb R}_+,X)$
denote the operator which maps $v$ to the piecewise constant function
$$
E_\tau v= u_n\quad\forall\, t\in (t_{n-1},t_n], \quad n=1,2,\dots
$$
The preceding two estimates imply
\begin{align}
&\|E_\tau v\|_{L^\infty({\mathbb R}_+,D(A))}\le c\|v\|_{D(A)} , \label{EtauvLinfty} \\
&\|E_\tau v\|_{L^{1/\alpha,\infty}({\mathbb R}_+,D(A))}\le c  \|v\|_X .  \label{EtauvLinfty2}
\end{align}
The estimate \eqref{EtauvLinfty} implies the first assertion of Theorem \ref{thm:Euler-inh}
in the case $p=\infty$, and the estimate \eqref{EtauvLinfty2} implies the second assertion of Theorem \ref{thm:Euler-inh}.
The real interpolation of the last two estimates yields
\begin{align*}
&\|E_\tau v\|_{(L^{1/\alpha,\infty}({\mathbb R}_+,D(A)),L^\infty({\mathbb R}_+,D(A)))_{1-\frac{1}{\alpha p},p}}
\le c\|v\|_{(X,D(A))_{1-\frac{1}{\alpha p},p}} ,
\,\,\, \forall\, p \in(\alpha^{-1}, \infty) .
\end{align*}
Since $(L^{1/\alpha,\infty}({\mathbb R}_+,D(A)),L^\infty({\mathbb R}_+,D(A)))_{1-\frac{1}{\alpha p},p}=L^p({\mathbb R}_+,D(A))$ \cite[Theorem 5.2.1]{BL1976}, this implies the first assertion of Theorem \ref{thm:Euler-inh} in the case $p\in(1/\alpha,\infty)$.
\qed\end{proof}

\begin{remark}
The proof shows that in the absence of a source term $f$, the maximal $\ell^p$-regularity of \eqref{Euler-inh}
only requires the sectorial property of $A$, rather than the $R$-sectorial property.
The general case (with nonzero source and nonzero initial data) is a linear combination of \eqref{eqn:fde} and \eqref{eqn:inh}.
\end{remark}

\begin{remark}
We have focused our discussions on the Caputo fractional derivative, since it allows
specifying the initial condition as usual, and thus is very popular among practitioners.
In the Riemann-Liouville case, generally it requires integral type initial condition(s)
\cite{KilbasSrivastavaTrujillo:2006}, for which the physical interpretation seems unclear.
\end{remark}

In the proof of Theorem \ref{thm:Euler-inh}, we first prove two end-point cases, $p=1/\alpha$ and $p=\infty$.
Then we use real interpolation method for the case $1/\alpha<p<\infty$.
The real interpolation method also holds for $0<p<1$ (\cite[Theorem 5.2.1]{BL1976}).
Thus, we have the following theorem in the case $1<\alpha<2$. The proof is omitted, since it
is almost identical with the proof of Theorem \ref{thm:Euler-inh}.

\begin{theorem}\label{thm:Euler-inh2}
Let $X$ be a Banach space, $1<\alpha<2$, and let $A$ be a sectorial operator on $X$ of angle $\alpha\pi/2$.
Then the BE scheme \eqref{Euler-inh} has the following maximal $\ell^p$-regularity:
\begin{align*}
&\|(\bar\partial_\tau^\alpha u^n)_{n=1}^N\|_{L^p(X)}+ \|(Au^n)_{n=1}^N\|_{L^p(X)} \\
&\le
\left\{
\begin{aligned}
&c_{p} (\|v\|_{(X,D(A))_{1-\frac{1}{p\alpha},p}}+\|w\|_{X}) ,
&&\mbox{for}\,\,\, p\in\Big[1,\frac{1}{\alpha-1}\Big), \\
&c_{p} (\|v\|_{(X,D(A))_{1-\frac{1}{p\alpha},p}}+\|w\|_{(X,D(A))_{1-\frac{1}{\alpha}-\frac{1}{p\alpha},p}}) ,
&&\mbox{for}\,\,\, p\in\Big(\frac{1}{\alpha-1},\infty\Big],
\end{aligned}
\right.
\end{align*}
and
\begin{align*}
\|(\bar\partial_\tau^\alpha u^n)_{n=1}^N\|_{L^{p,\infty}(X)}&+ \|(Au^n)_{n=1}^N\|_{L^{p,\infty}(X)} \\
&\le
c_{p} (\|v\|_{(X,D(A))_{1-\frac{1}{p\alpha},p}}+\|w\|_{X}) ,
&\mbox{for}\,\,\, p=\frac{1}{\alpha-1} ,
\end{align*}
where the constant $c_{p}$ depends on the bound of the set of operators
$\{zR(z;A) :z\in \Sigma_{\alpha\pi/2}\}$,
independent of $N$, $\tau$ and $A$.
\end{theorem}

\section{Examples and application to error estimates}\label{sec:example}

In this section, we present a few examples of fractional evolution equations which possess the maximal $L^p$-regularity,
and investigate conditions under which the time-stepping schemes in Sections \ref{sec:CQ}-\ref{sec:CN} satisfy the maximal $\ell^p$-regularity.

\begin{example}\label{Exp-continuous}
{\upshape ({\bf Continuous problem})
Consider the following time fractional parabolic equation in a bounded smooth domain
$\Omega\subset{\mathbb R}^d$ ($d\ge 1$):
\begin{align}\label{eqn:frac-continuous}
    \left\{
    \begin{aligned}
    &\partial_t^\alpha u(x,t)  = \Delta u(x,t) + f(x,t)
    &&\mbox{for}\,\,\,(x,t)\in \Omega\times(0,T) ,\\
    &u(x,t)   = 0 &&\mbox{for}\,\,\,(x,t)\in\partial\Omega\times(0,T) ,\\
    &u(x,0)   = 0 &&\mbox{for}\,\,\,x\in \Omega ,\quad \mbox{if\,\,\,$0<\alpha<1$,}\\
    &u(x,0) =\partial_t u(x,0)   = 0 &&\mbox{for}\,\,\,x\in \Omega ,\quad \mbox{if\,\,\,$1<\alpha<2$,}
    \end{aligned}
    \right.
\end{align}
where $T>0$ is given and $\Delta$ denotes the Laplacian operator.
In the appendix, we show that the $L^q$ realization $\Delta_q$ in
$X=L^q(\Omega)$ of $\Delta$ is an $R$-sectorial operator in $X$ with angle $\theta\in(0,\pi)$, and
that $\Delta_q v$ coincides with $\Delta v$ in the domain $D(\Delta_q)$ of $\Delta_q$:
\begin{align}\label{Deltaqv=Deltav0}
\Delta_q v=\Delta v,
\quad\forall\, v\in D(\Delta_q) ,\quad\forall\, 1<q<\infty .
\end{align}
Thus Theorem \ref{thm:maximal-fde} implies that the solution $u_q$ of
\begin{align}\label{eqn:frac-continuous-p}
    \left\{
    \begin{aligned}
    &\partial_t^\alpha u_q   = \Delta_q u_q  + f ,\\
    &u_q(\cdot,0)   = 0 &&\mbox{if\,\,\,$0<\alpha<1$,}\\
    &u_q(\cdot,0) =\partial_t u_q(\cdot,0)   = 0 &&\mbox{if\,\,\,$1<\alpha<2$,}
    \end{aligned}
    \right.
\end{align}
satisfies $u_q(\cdot,t)\in D(\Delta_q)$ for almost all $t \in\mathbb{R}_+$ and
\begin{align}\label{stinuous-p}
\begin{aligned}
&\quad\|\partial_t^\alpha u_q \|_{L^p(0,T;L^q(\Omega))}
  + \|\Delta_q u_q\|_{L^p(0,T;L^q(\Omega))} \\
&\le \|\partial_t^\alpha u_q \|_{L^p({\mathbb R}_+;L^q(\Omega))}
  + \|\Delta_q u_q\|_{L^p({\mathbb R}_+;L^q(\Omega))} \\
&\le c_{p,X}\|f\|_{L^p({\mathbb R}_+;L^q(\Omega))} \\
&= c_{p,X}\|f\|_{L^p(0,T;L^q(\Omega))},\qquad
  \forall\, 1<p,q<\infty .
\end{aligned}
\end{align}

In view of \eqref{Deltaqv=Deltav0}, we shall denote $(\Delta_q,D(\Delta_q))$ by $(\Delta,D_q(\Delta))$ below.
Then \eqref{Deltaqv=Deltav0}-\eqref{stinuous-p} imply that
for any given $1<p,q<\infty$ and $f\in L^p(0,T;L^q(\Omega))$, problem \eqref{eqn:frac-continuous} has a unique solution
$u\in L^p(0,T;D_q(\Delta))\cap W^{1,p}(0,T;L^q(\Omega))$  satisfying the maximal regularity
\begin{align*}
\|\partial_t^\alpha u \|_{L^p(0,T;L^q(\Omega))}
  + \|\Delta u\|_{L^p(0,T;L^q(\Omega))}
\le
c_{p,X}\|f\|_{L^p(0,T;L^q(\Omega))} .
\end{align*}
}\end{example}

\begin{example}
{\upshape ({\bf Time discretization})
Since the Dirichlet Laplacian $\Delta:D_q(\Delta)\rightarrow L^q(\Omega)$ defined in Example
\ref{eqn:frac-continuous} is $R$-sectorial of angle $\theta$ for all $\theta\in(0,\pi)$,
Theorems \ref{thm:Euler}, \ref{thm:SBD} and \ref{thm:L1}
imply that the time (semi-)discrete solutions given by the backward Euler,
BDF2 and L1 scheme satisfy
the following maximal $\ell^p$-regularity:
\begin{equation}\label{Max-reg-Laplacian}
  \|(\bar\partial_\tau^\alpha u^n)_{n=1}^N\|_{\ell^p(L^q(\Omega))}+ \|(\Delta u^n)_{n=1}^N\|_{\ell^p(L^q(\Omega))}\leq c_{p,q}\|(f^n)_{n=0}^N\|_{\ell^p(L^q(\Omega))} .
\end{equation}
By Theorem \ref{THM:CN-1}, the fractional Crank--Nicolson solution also satisfies
\eqref{Max-reg-Laplacian} when $0<\alpha<1$. Since $\Delta$ is self-adjoint and has an unbounded
spectrum, it follows that $r(\Delta)=\infty$, so the conditions of Theorems \ref{THM:Explicit}
and \ref{THM:CN-2} cannot be satisfied.
}
\end{example}

\begin{example}\label{Exp-continuous}
{\upshape({\bf Space-time fractional PDE})
Consider the following space-time nonlocal parabolic equation in ${\mathbb R}^d$ ($d\ge 1$):
\begin{align}\label{eqn:Laplace-half}
    \left\{
    \begin{aligned}
    &\partial_t^\alpha u(x,t)  =-(-\Delta)^\frac{1}{2} u(x,t) + f(x,t)
    &&\mbox{for}\,\,\,(x,t)\in{\mathbb R}^d\times{\mathbb R}^+,\\
    &u(x,0)   = 0 &&\mbox{for}\,\,\,x\in{\mathbb R}^d ,\quad \mbox{if\,\,\,$0<\alpha<1$,}\\
    &u(x,0) =\partial_t u(x,0)   = 0 &&\mbox{for}\,\,\,x\in{\mathbb R}^d ,\quad \mbox{if\,\,\,$1<\alpha<2$,}
    \end{aligned}
    \right.
\end{align}
where the fractional Laplacian $(-\Delta)^\frac{1}{2}v$ is defined by
$$
(-\Delta)^\frac{1}{2}v:={\mathcal F}^{-1}_\xi\big(|\xi|({\mathcal F}v)(\xi)\big),
\quad\forall\, v\in   W^{1,q}({\mathbb R}^d) .
$$
For $X:=L^q({\mathbb R}^d)$ and
$D_q((-\Delta)^\frac{1}{2} )
:= W^{1,q}({\mathbb R}^d)$, $1< q<\infty$,
the fractional operator
$-(-\Delta)^\frac{1}{2}:W^{1,q}({\mathbb R}^d)\rightarrow L^q({\mathbb R}^d)$ is
also $R$-sectorial of angle $\theta$ for arbitrary $\theta\in(0,\pi)$
\cite[proof of Proposition 2.2]{AkrivisLi2017}.
Hence, by Theorems \ref{thm:Euler}, \ref{thm:SBD} and \ref{thm:L1}, the backward Euler,
BDF2 and L1 schemes all satisfy the following
maximal $\ell^p$-regularity when $0<\alpha<2$ and $\alpha\neq 1$
\begin{equation*}
  \|(\bar\partial_\tau^\alpha u^n)_{n=1}^N\|_{\ell^p(L^q(\Omega))}+ \|((-\Delta)^\frac{1}{2} u^n)_{n=1}^N\|_{\ell^p(L^q(\Omega))}\leq c_{p,q}\|(f^n)_{n=0}^N\|_{\ell^p(L^q(\Omega))} .
\end{equation*}
By Theorem \ref{THM:CN-1}, the fractional Crank--Nicolson scheme also satisfies this estimate when $0<\alpha<1$.
}
\end{example}

\begin{example}\label{Exp-continuous}
{\upshape({\bf Fractional PDEs with complex coefficients})
Consider the following time-fractional PDE with a complex coefficient in a bounded Lipschitz domain
$\Omega\subset{\mathbb R}^d$ ($d\ge 1$):
\begin{align}\label{eqn:Laplace}
    \left\{
    \begin{aligned}
    &\partial_t^\alpha u(x,t) -e^{{\rm i}\varphi}\Delta u(x,t) =f(x,t)
    &&\mbox{for}\,\,\,(x,t)\in \Omega\times{\mathbb R}^+,\\
    & u (x,t) = 0 && \mbox{for}\,\,\, (x,t)\in \partial\Omega\times \mathbb{R}^+,\\
    &u(x,0)   = 0 &&\mbox{for}\,\,\,x\in \Omega ,\quad \mbox{if\,\,\,$0<\alpha<1$,}\\
    &u(x,0) =\partial_t u(x,0)   = 0 &&\mbox{for}\,\,\,x\in \Omega ,\quad \mbox{if\,\,\,$1<\alpha<2$,}
    \end{aligned}
    \right.
\end{align}
where $\varphi\in(-\pi,\pi)$ is given. It is worth noting that if $\varphi\in(\pi/2,\pi)\cup (-\pi,-\pi/2)$,
then \eqref{eqn:Laplace} is a diffusion-wave problem, since the operator $-e^{{\rm i}\varphi}\Delta$
has eigenvalues with negative real part. For $X:=L^q(\Omega)$ and
$D_q(e^{{\rm i}\varphi}\Delta):=D_q(\Delta)$, $1< q<\infty$,
the operator $e^{{\rm i}\varphi}\Delta :D_q(\Delta)\rightarrow L^q(\Omega)$ is
$R$-sectorial of angle $\theta$ for arbitrary $\theta\in(0,\pi-\varphi)$.
Hence, by Theorems \ref{thm:Euler}, \ref{thm:SBD} and \ref{thm:L1}, the backward Euler, BDF2 and L1 schemes satisfy the
maximal $\ell^p$-regularity estimate \eqref{Max-reg-Laplacian} when
$0<\alpha<2-2\varphi/\pi$ and $\alpha\neq 1$; the fractional Crank--Nicolson
scheme also satisfies the estimate \eqref{Max-reg-Laplacian} when $0<\alpha<\min(2-2\varphi/\pi,1)$.
}
\end{example}

As an application of the maximal $\ell^p$-regularity,
we present error estimates for the numerical solutions
by the BE scheme \eqref{eqn:Euler}, with the scaled $L^p$-norm \eqref{Lpnorm}.
Other time-stepping schemes can be analyzed similarly.
\begin{theorem}\label{THM:Error-estimate}
Let $A:D(A)\rightarrow X$ be an $R$-sectorial operator of angle $\alpha\pi/2$, with
$\alpha\in(0,2)$ and $\alpha\neq 1$, and the solution $u$ of \eqref{eqn:fde} be
sufficiently smooth. Then the solution of the BE scheme \eqref{eqn:Euler} satisfies for any $1<p<\infty$
\begin{equation}\label{Error-Estimate:BE}
  \|\bar\partial_\tau^\alpha (u^n- u(t_n))_{n=1}^N \|_{L^p(X)}
  + \|A(u^n-u(t_n))_{n=1}^N\|_{L^p(X)}
  \leq c_{p}\, \tau .
\end{equation}
\end{theorem}
\begin{proof}$\,\,$
We denote by $e^n:=u^n-u(t_n)$ the error of the numerical solution $u^n$.
Then $e^n$ satisfies
\begin{equation}\label{eqn:Error-BE}
\bar\partial_\tau^\alpha e^n = A e^n -E^n,  \quad n=1,2,\dots \\
\end{equation}
with $e^0=0$, where $E^n:=\bar\partial_\tau^\alpha u(t_n)-\partial_t^\alpha u(t_n)$
denotes the truncation error, satisfying
$\displaystyle\max_{1\le n\le N}\|E^n\|_{X} \le c\tau  $ \cite{Gorenflo:1995}.
By applying Theorem \ref{thm:Euler} to \eqref{eqn:Error-BE},
we obtain
\begin{equation*}
  \|(\bar\partial_\tau^\alpha e^n)_{n=1}^N\|_{L^p(X)} + \|(Ae^n)_{n=1}^N\|_{L^p(X)}\leq c_{p,X}c_R\|(E^n)_{n=1}^N\|_{L^p(X)}
  \le c_p\, \tau .
\end{equation*}\\[-25pt]
\qed\end{proof}

If $\Omega$ is a bounded smooth domain,
$X=L^q(\Omega)$, $1<q<\infty$, $D(A)=W^{2,q}(\Omega)\cap W^{1,q}_0(\Omega)$ and $A=\Delta$ (the Dirichlet Laplacian),
then the conditions of Theorem \ref{THM:Error-estimate} are satisfied, provided that the solution $u$ is smooth,
and \eqref{Error-Estimate:BE} gives that for any $1<p,q<\infty$
\begin{equation*}
\|(u^n-u(t_n))_{n=1}^N\|_{L^p(W^{2,q}(\Omega))}
\le c_{q}\|\Delta (u^n-u(t_n))_{n=1}^N\|_{L^p(L^q(\Omega))}
  \leq c_{p,q}\, \tau.
\end{equation*}
When $q>d$, error estimates in such strong norms as $W^{2,q}(\Omega)\hookrightarrow W^{1,\infty}(\Omega)$ can be used to
control some strong nonlinear terms in the numerical analysis of nonlinear parabolic problems
\cite{AkrivisLi2017,AkrivisLiLubich2016,Geissert:2007}. We will explore such an analysis in the future.


\appendix
\renewcommand{\thelemma}{A.\arabic{lemma}}
\renewcommand{\thetheorem}{A.\arabic{theorem}}
\renewcommand{\theequation}{A.\arabic{equation}}
\setcounter{equation}{0}

\section*{Appendix: $R$-sectorial property of $\Delta_q$}\label{app:R-sectorial}
In this appendix, we show that the $L^q$ realization $\Delta_q$ in
$X=L^q(\Omega)$ of $\Delta$ is an $R$-sectorial operator in $X$ with
angle $\theta\in(0,\pi)$ (see also \cite[Lemma 8.2]{AkrivisLiLubich2016} for
related discussions).


Let $\Delta_2$ be the restriction of the operator $\Delta$ to the domain $D(\Delta_2)=\{v\in H^1_0(\Omega):\Delta v\in L^2(\Omega)\}$.
Then the densely defined self-adjoint operator $\Delta_2: D(\Delta_2)\rightarrow L^2(\Omega)$ generates a bounded analytic semigroup $E_2(t):L^2(\Omega)\rightarrow L^2(\Omega)$ \cite[Example 3.7.5]{ABHN}, which extends to a bounded analytic semigroup $E_q(t)$ on $L^q(\Omega)$, $1<q<\infty$ \cite[Theorem 3.1]{Ouhabaz1995}, such that
\begin{align}
&E_{q_1}(t)v=E_{q_2}(t)v, &&\forall\, v\in L^{q_1}(\Omega)\cap L^{q_2}(\Omega), \label{Eq1=Eq2}\\
&E_q(t)v=\int_\Omega G(t,x,y)v(y) dy ,&&\forall\, v\in L^q(\Omega) ,\nonumber
\end{align}
where $G(t,x,y)$ is the kernel of the semigroup $E_2(t)$, i.e., the parabolic Green's function. It satisfies the
following Gaussian estimate \cite[Corollary 3.2.8]{Davies1989}:
\begin{align}\label{gaussian}
0 \le G(t,x,y) \le ct^{-\frac{d}{2}}e^{-\frac{|x-y|^2}{ct}} .
\end{align}
Let $\Delta_q$ be the generator of the semigroup $E_q(t)$, with its domain
\cite[Proposition 3.1.9, g]{ABHN}
\begin{align}\label{Def-DDeltaq}
D(\Delta_q)=\Big\{v\in L^q(\Omega): \lim_{t\downarrow 0}\frac{E_q(t)v-v}{t}\,\,\mbox{exists in $L^q(\Omega)$}\Big\}.
\end{align}
\eqref{Eq1=Eq2} and \eqref{Def-DDeltaq} imply that
\begin{align*}
&D(\Delta_{q_2})\subset D(\Delta_{q_1})
&&\mbox{for $1<q_1<q_2<\infty$},\\
&\Delta_{q_1}v=\Delta_{q_2}v &&\mbox{for $v\in D(\Delta_{q_2})\cap D(\Delta_{q_1})$} .
\end{align*}
In particular, we have
\begin{align}\label{Deltaq-Delta}
\Delta_qv=\Delta v,\quad\forall\, v\in D(\Delta_q)\cap D(\Delta_2),
\quad\forall\, 1<q<\infty.
\end{align}
The Gaussian estimate \eqref{gaussian} yields
\begin{align*}
\|E_q(t/2)v\|_{L^2(\Omega)}\le ct^{-\frac{d}{2}}\|v\|_{L^1(\Omega)}
\le ct^{-\frac{d}{2}}\|v\|_{L^q(\Omega)},\quad\forall\, v\in L^q(\Omega) .
\end{align*}
That is, $E_q(t/2)v\in L^2(\Omega)$ for $v\in L^q(\Omega)$ and $t>0$.
Hence, \eqref{Eq1=Eq2} implies
\begin{align}\label{EqinD2}
E_q(t)v=E_q(t/2)E_q(t/2)v=E_2(t/2)E_q(t/2)v\in D(\Delta_2) ,
\end{align}
where the last inclusion is due to the analyticity of the semigroup $E_2(t)$ \cite[Theorem 3.7.19]{ABHN}.
Then \eqref{Deltaq-Delta} and \eqref{EqinD2} imply
\begin{align*}
\lim_{t\downarrow 0}\|\Delta E_q(t) v-\Delta_q v\|_{L^q(\Omega)}
&= \lim_{t\downarrow 0} \|\Delta_q E_q(t) v-\Delta_q v\|_{L^q(\Omega)}
= 0, \quad\forall\, v\in D(\Delta_q) .
\end{align*}
Since $\displaystyle \lim_{t\downarrow 0}\|E_q(t) v-v\|_{L^q(\Omega)}=0$,
the last identity implies
\begin{align*}
(\Delta_q v,\varphi)=
\lim_{t\downarrow 0}(\Delta E_q(t) v,\varphi)
=\lim_{t\downarrow 0}(E_q(t) v,\Delta \varphi)
=(v,\Delta \varphi) ,
\quad\forall\, v\in D(\Delta_q),\,\,\,\forall\,\varphi\in C^\infty_0(\Omega).
\end{align*}
That is, $\Delta_q v$ coincides with the distributional partial derivative $\Delta v$ in the sense of distributions, i.e.,
\begin{align}\label{Deltaqv=Deltav}
\Delta_q v=\Delta v,
\quad\forall\, v\in D(\Delta_q) ,\quad\forall\, 1<q<\infty .
\end{align}

\begin{remark}$\,$
If the domain $\Omega$ is smooth or convex, then we have the characterization
\begin{align*}
D(\Delta_q)=\{v\in W^{1,q}_0(\Omega):
\Delta v\in L^q(\Omega)\} .
\end{align*}
However, this characterization does not hold in general bounded Lipschitz domains (e.g., nonconvex polygons).
In a general bounded Lipschitz domain, the operator $\Delta_2^{-1}:L^2(\Omega)\rightarrow L^2(\Omega)$ has an extension
$\Delta^{-1}:L^1(\Omega)\rightarrow L^1(\Omega)$, given by \cite{Gruter1982}
$$
\Delta^{-1}v(x)=\int_\Omega {\mathcal G}(x,y)v(y) dy
$$
in terms of the elliptic Green's function ${\mathcal G}(x,y)$,
satisfying
\begin{align*}
&\Delta^{-1}v=\Delta_q^{-1}v,
\quad \forall\, v\in L^q(\Omega).
\end{align*}
Hence, we have the characterization $D(\Delta_q)=\{\Delta^{-1}v:v\in L^q(\Omega)\}$.
\end{remark}

\bibliographystyle{abbrv}
\bibliography{frac}



\end{document}